\documentclass[a4paper, 11pt]{article}
\usepackage[utf8]{inputenc}
\usepackage{amsmath,amssymb,amsfonts,amsthm,mathscinet}
\usepackage{hyperref}
\usepackage{enumerate}
\usepackage{authblk}
\usepackage{color,graphicx}

\newtheorem{theorem}{Theorem}
\newtheorem{lemma}[theorem]{Lemma}
\newtheorem{proposition}[theorem]{Proposition}

\newtheorem{definition}[theorem]{Definition}
\theoremstyle{remark}
\newtheorem{remark}{Remark}

\newcommand{\Z}{\mathbb{Z}}
\newcommand{\Zp}{\Z^+}
\newcommand{\R}{\mathbb{R}}
\newcommand{\C}{\mathbb{C}}
\newcommand{\D}{\partial_t^\alpha}
\newcommand{\Dg}{\Delta_g}

\newcommand{\bra}{\langle}
\newcommand{\ket}{\rangle}
\newcommand{\dom}{\mathcal{D}}
\newcommand{\LM}{{L^2(M)}}
\newcommand{\LV}{{L^2(V)}}
\newcommand{\dtau}{\,d\tau}
\newcommand{\inv}{^{-1}}
\newcommand{\LSS}{\LSSop_{V}}
\newcommand{\LSST}{\LSSop_{V,T}}
\newcommand{\LSSTilde}{\widetilde{\LSSop}_{\widetilde{V}}}
\newcommand{\LSSTTilde}{\widetilde{\LSSop}_{\widetilde{V},T}}

\DeclareMathOperator{\supp}{supp}
\DeclareMathOperator{\linspan}{span}
\DeclareMathOperator{\re}{Re}
\DeclareMathOperator{\cl}{cl}

\DeclareMathOperator{\LSSop}{L}

\title{Inverse problems for heat equation and space--time fractional diffusion equation with one measurement}
\author[1]{Tapio Helin\thanks{tapio.helin@lut.fi}}
\author[2]{Matti Lassas\thanks{matti.lassas@helsinki.fi}}
\author[2]{Lauri Ylinen\thanks{lauri.ylinen@helsinki.fi}}
\author[2]{Zhidong Zhang\thanks{zhidong.zhang@helsinki.fi}}
\affil[1]{School of Engineering Science, LUT University, Finland}
\affil[2]{Department of Mathematics and Statistics, University of Helsinki, Finland}

\begin{document}
\maketitle

\begin{abstract}
  \noindent
  Given a connected compact Riemannian manifold $(M,g)$ without
  boundary, $\dim M\ge 2$, we consider a space--time fractional
  diffusion equation with an interior source that is supported on an
  open subset $V$ of the manifold. The time-fractional part of the
  equation is given by the Caputo derivative of order $\alpha\in(0,1]$,
  and the space fractional part by $(-\Dg)^\beta$, where
  $\beta\in(0,1]$ and $\Dg$ is the Laplace--Beltrami operator on the
  manifold. The case $\alpha=\beta=1$, which corresponds to the
  standard heat equation on the manifold, is an important special
  case. We construct a specific source such that measuring the
  evolution of the corresponding solution on $V$ determines the
  manifold up to a Riemannian isometry.

  \text{\bf{Keywords}}: inverse problem, space--time fractional
  diffusion
  equation, regularity, uniqueness
\end{abstract}

\section{Introduction}

\subsection{Statement of the problem and main results}

Throughout this paper, $(M,g)$ will denote a connected compact smooth
Riemannian manifold without boundary, with metric $g$ and
$\dim M\ge 2$, and $V\subset M$ will be a nonempty open subset with
smooth boundary. Also, $0<\alpha\le 1$ and $0<\beta\le 1$ will be
fixed parameters. An important special case that we also consider is
the heat equation that corresponds to $\alpha=\beta=1$. We note that to the knowledge of the authors,  the main results (Theorems~\ref{thm:main-theorem} and~\ref{thm:M-equality}, and Proposition~\ref{prop:LSS-uniqueness}) are new also in this special case.

We consider the following space--time fractional diffusion equation:
\begin{subequations}
  \label{eq:FDE}
  \begin{align}
    \D u(x,t) + (-\Dg)^\beta u(x,t) &= f(x,t),  && (x,t)\in M\times (0,\infty),\label{eq:FDE-a}\\
    u(x,0) &= 0, && x\in M.
  \end{align}
\end{subequations}
Here the source term $f$ is supported on $V\times(0,T)$, for some
$T>0$, and $\D$ is the Caputo (also known as the Djrbashian--Caputo) fractional-derivative  of order
$\alpha$. For a smooth function $y$ defined on $[0,\infty)$, the
Caputo fractional-derivative is defined by $\partial^1_t y=y'$ for
$\alpha=1$, and
\begin{equation}
  \D y(t) :=
  \frac{1}{\Gamma(1-\alpha)}\int_0^t (t-\tau)^{-\alpha}y'(\tau)\ d\tau
  \qquad(t\ge 0,\,0<\alpha<1),
\end{equation}
where $\Gamma$ is the Euler's gamma function. In the second term
of~\eqref{eq:FDE-a}, $\Dg$ is the Laplace--Beltrami operator, and the
fractional power is taken in the sense of functional calculus. The precise
definition of $\D$ and $(-\Dg)^\beta$ can be found in
Section~\ref{sec:preliminaries}.

We show that for a smooth compactly supported source $f$ there exists
a unique so-called strong solution $u^f$ of~\eqref{eq:FDE}. The local
source-to-solution operator $\LSS$ is then defined as the operator
\begin{equation*}
  f\mapsto\LSS f:=u^f|_{V\times[0,\infty)}.
\end{equation*}

In this paper we consider an inverse problem for the space--time
fractional diffusion equation~\eqref{eq:FDE}, namely does $\LSS$
determine the manifold $(M,g)$ uniquely? Note that the input $f$ to
the local source-to-solution operator $\LSS$, i.e., the source term in
equation~\eqref{eq:FDE-a}, is supported on $V$. Also, the value
$\LSS f$, i.e., the evolution of the solution $u^f$ of~\eqref{eq:FDE},
is observed only on $V$. Hence $\LSS$ is determined by information
residing on $V$ only.

We show that $\LSS$ indeed determines $(M,g)$ up to an isometry. In
fact, we show the stronger result that we do not need to know the
operator $\LSS$ completely, but it is enough to know the value of
$\LSS h$ on some nonempty time interval $[0,T)$ with only one source
$h$, provided the source $h$ is chosen appropriately. Below, $\cl(V)$
denotes the closure of the set $V$.

\begin{theorem}
  \label{thm:main-theorem}
  Let $(M,g)$ be a connected compact smooth Riemannian manifold
  without boundary, with metric $g$ and $\dim M\ge 2$, let
  $V\subset M$ be a nonempty open subset with smooth boundary and let
  $T>0$. Then it is possible to construct a source
  $h\in C_c^\infty((0,T);\LV)$ such that the data
  \begin{equation*}
    (V,u^h|_{V\times[0,T)})
  \end{equation*}
  determines the manifold $(M,g)$ up to a Riemannian isometry. More
  precisely this means the following:
  
  Let $(\widetilde{M},\widetilde{g})$ be another smooth, connected and
  compact Riemannian manifold without boundary, with metric
  $\widetilde{g}$, and let $\widetilde{V}\subset\widetilde{M}$ be an
  open and nonempty set with smooth boundary. Then it is possible to
  construct a source $h\in C_c^\infty((0,T);\LV)$ that has the
  following property: If there exists a diffeomorphism
  \begin{equation*}
    \theta : \cl(\widetilde{V})\to\cl(V)
  \end{equation*}
  such that the solutions $u^h$ of~\eqref{eq:FDE} with source $h$ and
  the solution $\widetilde{u}^{\theta^* h}$ of the corresponding
  equation on $(\widetilde{M},\widetilde{g})$ with source $\theta^* h$
  satisfy
  \begin{equation}
    \label{eq:main-theorem}
    (\theta^* u^h)|_{\widetilde{V}\times[0,T)} = (\widetilde{u}^{\theta^* h})|_{\widetilde{V}\times[0,T)},
  \end{equation}
  then $(M,g)$ and $(\widetilde{M},\widetilde{g})$ are Riemannian
  isometric.
\end{theorem}
\begin{remark}
  \label{rem:pullback}
  The pullback $\theta^*$ of the diffeomorphism $\theta$ acts on an
  $\LV$-valued function $f$ by
  \begin{equation*}
    (\theta^* f)(t) := \theta^*(f(t)).
  \end{equation*}
  Above, an open set with smooth boundary refers in local coordinates to definition given in \cite[Appendix C.1]{evans2002partial}.
\end{remark}
\begin{remark}
  An explicit expression for the source $h$ is given in
  Definition~\ref{def:h}.
\end{remark}

In Section~\ref{sec:LSS} we show that the local source-to-solution
operator $L_V$ is well-defined as an operator
\begin{equation}
  \label{eq:LSS-first}
  \LSS:C^2_c((0,\infty);\LV)\to C^1([0,\infty);\LV)\cap L^\infty([0,\infty);\LV).
\end{equation}
As in~\eqref{eq:LSS-first}, instead of considering functions depending
on both the space variable $x\in M$ and the time variable $t\in\R$, it
is convenient to consider them as functions of time $t\in\R$ taking
values in the Hilbert space of square-integrable functions on $M$. For
the convenience of the reader, we review the necessary definitions and
results of calculus of Hilbert space valued functions in the Appendix.

The proof of Theorem~\ref{thm:main-theorem} consists of two parts. The
first part is to show that the function
$\LSS h|_{[0,T)}\in C^1([0,T);\LV)\cap L^\infty([0,T);\LV)$ uniquely
determines the operator $\LSS$. The second part is to show that the
operator $\LSS$ determines the manifold $(M,g)$ up to a Riemannian
isometry. These steps are formulated below as two independent results.

In the following, let $T>0$ be a constant, $(M,g)$ and
$(\widetilde{M},\widetilde{g})$ be Riemannian manifolds, $V\subset M$
and $\widetilde{V}\subset\widetilde{M}$ be open sets, and
$\theta:\cl(\widetilde{V})\to\cl(V)$ be a diffeomorphism, and assume
they all satisfy the assumptions of
Theorem~\ref{thm:main-theorem}. Furthermore, let
\begin{equation*}
  \LSSop_{\widetilde{V}}:C^2_c((0,\infty);L^2(\widetilde{V}))\to C^1([0,\infty);L^2(\widetilde{V}))\cap L^\infty([0,\infty);L^2(\widetilde{V}))
\end{equation*}
be the local source-to-solution operator on the manifold
$(\widetilde{M}, \widetilde{g})$.

\begin{proposition}
  \label{prop:LSS-uniqueness}
  Let $h\in C_c^\infty((0,T);\LV)$ be the source defined in
  Definition~\ref{def:h}.  If
  \begin{equation}
    \label{eq:LSSh-equality}
    (\theta^*\LSS h)|_{[0,T)} = (\LSSop_{\widetilde{V}} \theta^*h)|_{[0,T)},
  \end{equation}
  then
  \begin{equation}
    \label{eq:LSS-equality}
    \theta^*\LSS f= \LSSop_{\widetilde{V}} \theta^* f
  \end{equation}
  for every $f\in C^2_c((0,\infty);\LV)$.
\end{proposition}

\begin{theorem}
  \label{thm:M-equality}
  If the equality~\eqref{eq:LSS-equality} holds for every
  $f\in C^2_c((0,\infty);\LV)$, then the manifolds $(M,g)$ and
  $(\widetilde{M},\widetilde{g})$ are Riemannian isometric.
\end{theorem}

\subsection{Motivation and literature}
Einstein's celebrated paper \cite{Einstein:1905b} introduced the 
classical explanation of
Brownian motion as a random walk, in which the dynamics of a particle suspended in a fluid is described by an uncorrelated, Markovian, Gaussian stochastic process. 
A key result of this theory is that the mean-square displacement of the random walk is
proportional to time, i.e., $\langle x^2\rangle \propto t$ for large $t$. At the continuum limit, it follows that the concentration of a large number of 
independent particles is governed by the diffusion equation.

%$t$, i.e.  $\langle x^2\rangle = C t$, the
%diffusion equation can be derived.

%Mathematically we can generalize this case to the continuous time
%random walk with continuously distributed jumps, in which the waiting 
%time between two successive jumps, as well as the length of a given 
%jump follow the given probability density functions. We denote them 
%by $\psi(t)$ and $\lambda(x)$, which are also named the waiting time 
%distribution and the jump length distribution, respectively. Provided the 
%characteristic waiting time $T:=\int_0^\infty t\psi(t)\ dt$ and the jump length variance
%$\Sigma:=\int_{-\infty}^\infty x^2\lambda(x)\,dx$ are both finite, the
%Central Limit Theorem yields that the long-time limit again converges
%to the Brownian motion, which leads to the classical diffusion equation.

Despite the success of standard diffusion model, there are 
a number of experimental observations of diffusion processes, where
the mean-square displacement does not scale linearly.
A random walk interpretation can also be given to such processes:
In a standard discrete random walk the step length is a fixed distance and the steps occur at discrete times. In a more general walk (Continuous Time Random Walk, CTRW) a waiting time and step length are sampled from given probability distributions. 
At the continuum limit, a suitable power law distribution for the waiting time
results to {\it subdiffusive} processes, where
$\langle x^2\rangle \propto t^\alpha$, $0<\alpha < 1$. 
Analogous to the classical diffusion, the concentration of random particles
satisfies a model where the time derivative in the diffusion equation is replaced by a fractional time derivative of order $\alpha$. 
Similarly, a suitable power law step length distribution replaces the Laplacian in the diffusion equation by a fractional power $(-\Delta)^\beta$. The variability of these distributions gives rise to the class of fractional PDEs in \eqref{eq:FDE}.

Anomalous diffusion processes described by equation \eqref{eq:FDE} appear in spatially disordered systems such as porous media, in turbulent fluids and plasma, biological systems and finance (see, e.g.,
\cite{BagleyTorvik:1983,Koeller:1984,He:1998,
  SabatierAgrawalMachado:2007,Nigmatullin:1986,BerkowitzCortisDentzScher:2006,
  DelCarrerasLynch:2005,CarrerasLynchZaslavsky:2001,CarteaDel:2007,ZoiaRossaKardar:2007, 
  DubkovSpagnoloUchaikin:2008}).
Following the random walk analogy, our main result in Theorem \ref{thm:main-theorem} can be interpreted as follows: we introduce a rigorous strategy to inject new particles into a diffusion process taking place in an unknown medium so that a single long-term observation of the concentration determines the properties of the medium.

Mathematical work on fractional calculus is extensive. For a general overview, see textbooks \cite{kilbas2006theory,SamkoKilbasMarichev:1993}, reviews \cite{henry2010introduction, bucur16} and references therein.
Without providing a comprehensive list, we mention that
classical properties for the fractional diffusion equations, such as the fundamental solutions, the regularity estimates and the maximum principles are established in \cite{sakamoto2011initial, LuchkoYamamoto:2018,
  Luchko2009maximum,Liu2016strong,MainardiLuchkoPagnini:2001}.
Moreover, numerical analysis for fractional PDEs is considered in \cite{JiangMa:2011,JinLiZhou:2018,ChenLiuTurnerAnh:2007,YangTurnerLiuIli:2011,LiuZhuangAnhTurnerBurrage:2007}.

Inverse problems for fractional PDEs have gained major attention in recent years.
The review \cite{Jin2015tutorial} summarizes work on
some common fractional inverse problems and collects some open
problems. Uniqueness and reconstruction of unknown parameters are considered in
 \cite{RundellZhang:2018,Zhang:2017,LiuYamamoto:2010,
  SunYanWei:2019,SunZhangWei:2019}.
In particular, we mention the article \cite{KianOksanenSoccorsiYamamoto:2018} by Y. Kian, L. Oksanen, E. Soccorsi, and M. Yamamoto, where the uniqueness of the Riemannian metric is proved for time-fractional PDE given Dirichlet-to-Neumann map at a fixed time at the boundary of the manifold.
For techniques based on Carleman estimates, we refer to \cite{XuChengYamamoto:2011,HuangLiYamamoto:2019}. 

There are a number of other interesting setups for fractional inverse problems: In the static case, the fractional Calderon problems 
are investigated in \cite{RulandSalo:2018,GhoshLinXiao:2017,GhoshRulandSaloUhlmann:2018,Salo:2017,GhoshSaloUhlmann:2016}.
If a more general waiting time probability
distribution is considered, then $\partial^\alpha_t$ may need to be
replaced by a weighted mixture of fractional derivatives.
This leads to the so-called multi-term time
fractional diffusion equations and the distributed order differential
equations \cite{Rundell2017fractional,Li2017analyticity,LiLiuYamamoto:2015,
  LiYamamoto:2015, LiLuchkoYamamoto:2014}. Also, there is recent effort to study statistical fractional inverse problems  \cite{tuan2017inverse, guerngar2018inverse, nane2018approximate, ZhangJiaYan:2018, niu2018inverse}. 
  
The work in this paper is connected to geometric inverse problems outside fractional PDEs through many aspects of the observational setup. In wave propagation models with finite speed of propagation single measurement data has been studied in \cite{helin2012inverse, helin2014inverse}. The setup with multiple measurements is better understood: in such a case geometric version of boundary control method can be used for deriving uniqueness and reconstruction \cite{belishev1992reconstruction, belishev1988approach, katchalov2001inverse, anderson2004boundary, bingham2008iterative}. Finally, let us mention that closed manifolds have been studied also in the framework of inverse spectral problems \cite{krupchyk2008inverse}.

\subsection{Outline of the paper}

This paper is organized as follows. In Section~\ref{sec:preliminaries} we record some preliminary definitions and present
some well-known results regarding the Mittag--Leffler function $E_{a,b}$,
which plays a central role in representing the solution to~\eqref{eq:FDE}.

We investigate the direct problem for~\eqref{eq:FDE} in
Section~\ref{sec:direct-problem}, where the existence, uniqueness and representation
results of the solution are proved. Also, the source-to-solution
operator $\LSS$ is defined, which will be studied in the inverse problem part.

The inverse problem is considered in Section~\ref{sec:inverse-problem}. First, we prove that
the operator $\LSS$ can be uniquely determined given a single
measurement (Proposition~\ref{prop:LSS-uniqueness}). Second, we
prove that the operator $\LSS$ determines the manifold up to an
isometry (Theorem~\ref{thm:M-equality}). The main result,
Theorem~\ref{thm:main-theorem}, immediately follows from these two
results.

\section{Preliminaries}
\label{sec:preliminaries}

This section contains some technical tools that are required in
understanding equation~\eqref{eq:FDE}. We recall the definition of the
Caputo derivative of fractional order, and we review the
Laplace--Beltrami operator $\Dg$, as well as some basic functional
calculus to define its fractional powers. We also give a definition of
the strong solution of~\eqref{eq:FDE}.

%We begin by explaining the Mittag--Leffler function, which will be
%later used to represent solutions of equation~\eqref{eq:FDE}.

\subsection{The Mittag--Leffler function and fractional derivatives}

The (two-parameter) Mittag--Leffler function $E_{a,b}$ has a role in
the fractional differential equations analogous to the role of the
exponential function in the case of the integer order differential
equations. The function is defined as
\begin{equation}
  \label{eq:Mittag-Leffler}
  E_{a,b}(z) := \sum_{k=0}^\infty \frac{z^k}{\Gamma(ka+b)}
  \qquad(a>0,\,b>0,\,z\in\C).
\end{equation}
In particular, $E_{1,1}(z) = e^z$. For a treatise of the
Mittag--Leffler function, see~\cite{MR3244285}.

The radius of convergence of the power
series~\eqref{eq:Mittag-Leffler} is infinite, so $E_{a,b}$ is an
entire function. A recurrence relation for the gamma function together
with termwise differentiation of the power series shows that
$E_{a,1}'(z) = a\inv E_{a,a}(z)$. For every $\lambda\in\C$, the
function
\begin{equation}
  \label{eq:F-lambda}
  G_\lambda(z) := E_{a,1}(-\lambda z^a)
  \qquad (z\in\C_+:=\{z\in\C:\re(z)>0\})
\end{equation}
is holomorphic on $\C_+$, and therefore above reasoning shows that
\begin{equation}
  \label{eq:F-derivative}
  G_\lambda'(z) = -\lambda z^{a-1} E_{a,a}(-\lambda z^a)
  \qquad (z\in\C_+).
\end{equation}

\begin{proposition}
  \label{prop:ML-properties}
  For $0<a\le 1$, the following hold:
  \begin{enumerate}[(i)]
  \item There exists a constant $C_a>0$ (that depends on $a$) such
    that
    \begin{equation*}
      |E_{a,a}(z)| \le C_a
      \qquad(z\in\C\setminus\C_+).
    \end{equation*}
  \item Let $\lambda\ge 0$ and define a function
    $F_\lambda:(0,\infty)\to\C$ by
    \begin{equation*}
      F_\lambda(t):=t^{a-1}E_{a,a}(-\lambda t^a)
      \qquad(t>0).
    \end{equation*}
    Then the Laplace transform $\mathcal{L}F_\lambda(s)$ of
    $F_\lambda$ exists at every point $s\in\C_+$, and
    \begin{equation}
      \label{eq:F-Laplace}
      \mathcal{L}F_\lambda(s) = \frac{1}{s^a +\lambda}
      \qquad(s\in\C_+).
    \end{equation}
  \end{enumerate}
\end{proposition}
\begin{proof}
  \begin{enumerate}
  \item For $a=1$ the boundedness is evident, because
    $E_{1,1}(z)=e^z$. For $0<a<1$, see Theorem~1.6
    in~\cite{MR1658022}.
  \item For $s\in\C$ with $\re s>\lambda^{1/a}$,
    formula~\eqref{eq:F-Laplace} is proved in~\cite{MR1658022} (cf.\
    formula~(1.80) there). By the boundedness of $E_{a,a}$ on
    $\C\setminus\C_+$, the Laplace transform of $F_\lambda$ exists on
    the whole half-plane $\C_+$. It follows from the uniqueness of
    analytic continuation that~\eqref{eq:F-Laplace} holds for every
    $s\in\C_+$.
  \end{enumerate}
\end{proof}

Recall that $0<\alpha\le 1$ and consider a complex-valued function
$y\in C^1([0,\infty))$. Here the space is the space of continuously
differentiable functions on $[0,\infty)$, with the derivative at the
left endpoint being the appropriate one-sided derivative. The
\emph{Caputo derivative of order $\alpha$} of $y$ at point
$t\in[0,\infty)$, denoted by $\D y(t)$, is defined as
\begin{equation}
  \label{eq:Caputo-derivative}
  \D y(t) :=
  \begin{cases}
    \frac{1}{\Gamma(1-\alpha)}\int_0^t (t-\tau)^{-\alpha}
    y'(\tau)\dtau,&
    0<\alpha<1,\\
    y'(t), & \alpha=1.
  \end{cases}
\end{equation}
In particular, if $\alpha=1$, then $\D y$ is just the standard first
order derivative of $y$.

Another commonly used fractional derivative is the Riemann--Liouville
fractional derivative. The \emph{Riemann--Liouville fractional
  derivative of order $\alpha$} of $y\in C^1([0,\infty))$ at point
$t\in[0,\infty)$, denoted by $\partial^\alpha_{t,\mathrm{RL}} y(t)$,
is defined by
\begin{equation}
  \label{eq:RL-fractional-derivative}
  \partial^\alpha_{t,\mathrm{RL}} y(t)
  :=
  \begin{cases}
    \frac{1}{\Gamma(1-\alpha)}\frac{d}{dt}\int_0^t(t-\tau)^{-\alpha}y(\tau)\dtau,&
    0<\alpha<1,\\
    y'(t), & \alpha=1.
  \end{cases}
\end{equation}

It is clear from~\eqref{eq:RL-fractional-derivative} that the
Riemann--Liouville derivative can be defined for a larger class of
functions than the continuously differentiable ones. It can also be
shown that
\begin{equation}
  \label{eq:Caputo-RL-equivalence}
  \D y(t) = \partial^\alpha_{t,\mathrm{RL}} (y(t)-y(0))
  \qquad(t\ge 0),
\end{equation}
(see, e.g.,\ Chapter~3 in~\cite{MR2680847}), and
often~\eqref{eq:Caputo-RL-equivalence} is in fact taken as the
definition of the Caputo derivative, because the right-hand side
of~\eqref{eq:Caputo-RL-equivalence} is defined for a larger class of
functions than~\eqref{eq:Caputo-derivative}.

In this paper we mainly consider continuously differentiable
functions $y\in C^1([0,\infty))$ with $y(0)=0$. For such
functions~\eqref{eq:Caputo-RL-equivalence} shows that the Caputo
fractional derivative and the Riemann--Liouville fractional derivative
coincide. For consistency of notation, we use the Caputo fractional
derivative $\D$ throughout the paper.

For a scalar nonhomogeneous linear fractional differential equation,
there are the following existence and uniqueness results 
(see~\cite{MR2680847}):
\begin{proposition}
  \label{prop:scalar-FDE}
  Let $0<\alpha\le 1$, $\lambda\in\R$, $b\in C^1_c((0,\infty))$, and
  consider the fractional differential equation
  \begin{subequations}
    \begin{align}
      \label{eq:scalar-FDE-problem}
      \D y(t) + \lambda y(t) &= b(t), && (t\ge 0),\\
      \label{eq:scalar-initial-value}
      y(0) &= 0.
    \end{align}
  \end{subequations}
  There exists a unique function $y\in C^1([0,\infty))$ for which
  equations~\eqref{eq:scalar-FDE-problem}
  and~\eqref{eq:scalar-initial-value} are valid.  This function can be
  represented as
  \begin{equation}
    \label{eq:FDE-variation-of-constants}
    y(t) = \int_0^t (t-\tau)^{\alpha-1}
    E_{\alpha,\alpha}(-\lambda (t-\tau)^\alpha) b(\tau)\dtau
    \qquad(t\ge 0).
  \end{equation}
\end{proposition}
\begin{proof}
  By Corollary~6.9 in~\cite{MR2680847}, there exists at most one
  continuously differentiable function for
  which~\eqref{eq:scalar-FDE-problem}
  and~\eqref{eq:scalar-initial-value} are valid.  Some standard
  properties of convolutions and the assumed regularity of $b$ imply
  that $y$ as defined by~\eqref{eq:FDE-variation-of-constants} is
  continuously differentiable. Theorem~{7.2} in ~\cite{MR2680847}
  states that this function satisfies~\eqref{eq:scalar-FDE-problem}
  and~\eqref{eq:scalar-initial-value}.
\end{proof}

Of course, in the case $\alpha=1$, the existence and uniqueness of a
solution to~\eqref{eq:scalar-FDE-problem}
and~\eqref{eq:scalar-initial-value} follow from the standard theory
of linear ordinary differential equations,
and~\eqref{eq:FDE-variation-of-constants} is just the variation of
parameters formula.

\subsection{The Laplace--Beltrami operator}
\label{sec:Laplace-Beltrami}

The Laplace--Beltrami operator $\Dg$ is an unbounded self-adjoint
operator on $\LM$ with domain of definition $\dom{(\Dg)}=H^2(M)$. The
operator is defined in local coordinates by
\begin{equation*}
  \Dg \xi := |g|^{-1/2}\partial_j(|g|^{1/2}g^{jk}\partial_k\xi)
  \qquad(\xi\in H^2(M)),
\end{equation*}
where $|g|$ is the determinant of the metric $g$ and $(g^{jk})$ is the
inverse matrix of $g$. Here and below, we use Einstein's summation convention and sum over indexes appearing as sub- and superindexes.

Let $0=\lambda_1<\lambda_2\le\lambda_3\le\cdots$ be the eigenvalues of
$-\Dg$, listed according to their multiplicities, and let
$(\varphi_k)_{k=1}^\infty$ be some complete orthonormal sequence of associated
eigenfunctions. The exponent $(-\Dg)^\beta$ of $(-\Dg)$ is then
defined by 
\begin{equation*}
  (-\Dg)^\beta\xi :=
  \sum_{k=1}^\infty \lambda_k^\beta\bra \xi,\varphi_k\ket_\LM\varphi_k
  \qquad(\xi\in\LM),
\end{equation*}
(see \cite{lions2012non}) with domain
\begin{equation*}
  \dom((-\Dg)^\beta) :=
  \left\{\xi\in\LM:
    \sum_{k=1}^\infty\lambda_k^{2\beta}|\bra\xi,\varphi_k\ket_\LM|^2<\infty\right\}
  = H^{2\beta}(M).
\end{equation*}

\subsection{Fractional derivatives of
  \texorpdfstring{$\LM$}{L\^{}2(M)}-valued functions and the strong
  solution of\texorpdfstring{~\eqref{eq:FDE}}{FDE}}

Let $0<T\le\infty$ and recall that $0<\alpha\le 1$. Let
$y\in C^1([0,T);\LM)$. The \emph{Caputo
  derivative of order $\alpha$} of the $\LM$-valued function $y$ at
point $t\in[0,T)$, denoted by $\D y(t)$, is defined analogously to the
scalar case:
\begin{equation}
  \label{eq:Caputo}
  \D y(t) :=
  \begin{cases}
    \frac{1}{\Gamma(1-\alpha)}\int_0^t (t-\tau)^{-\alpha}
    y'(\tau)\dtau,&
    0<\alpha<1,\\
    y'(t), & \alpha=1.
  \end{cases}
\end{equation}
Here the derivative $y'$ is in the sense of the derivative of an
$\LM$-valued function of a real variable, and the integral is in the
sense of Bochner. Note that the existence of the integral is
guaranteed by the assumption of continuous differentiability of $y$.

We are now ready to give a definition of a solution of the fractional
diffusion equation~\eqref{eq:FDE}. Let $f:[0,\infty)\to\LM$. We say
that a function $u\in C^1([0,\infty);\LM)$ is a \emph{strong solution
  of the fractional diffusion equation~\eqref{eq:FDE}}, if
\begin{enumerate}[(i)]
\item $u(0)=0$,
\item $u(t)\in\dom((-\Dg)^\beta)$ for every $t\ge 0$, and
\item $\D u(t) + (-\Dg)^\beta u(t) = f(t)$ for every $t\ge 0$.
\end{enumerate}

\section{Analysis of the direct problem}
\label{sec:direct-problem}

Here we prove an existence and uniqueness result for the fractional
diffusion equation~\eqref{eq:FDE}. We also establish a representation
of the solution, which will later be used in solving the inverse
problem. Furthermore, we define the local source-to-solution operator
$\LSS$.

\subsection{Uniqueness and existence of a strong solution}
\label{sec:FDE-solution}

Let $0=\lambda_1<\lambda_2\le\lambda_3\le\cdots$ be the eigenvalues of
$-\Dg$, listed according to their multiplicities, and let
$(\varphi_k)_{k=1}^\infty\subset C^\infty(M)$ be some complete
orthonormal sequence of corresponding eigenfunctions.

\begin{proposition}
  \label{prop:existence-and-uniqueness}
  Suppose that $f\in C_c^2((0,\infty);\LM)$. Then there exists a
  unique strong solution $u\in C^1([0,\infty);\LM)$ of the fractional
  diffusion equation~\eqref{eq:FDE}. The strong solution can be
  represented as
  \begin{equation}
    \label{eq:u}
    u(t) = \sum_{k=1}^\infty u_k(t)\,\varphi_k
    \qquad (t\ge 0),
  \end{equation}
  where the series converges in $\LM$ for every $t\ge 0$, and
  \begin{equation}
    \label{eq:u-k}
    u_k(t) := \int_0^t (t-\tau)^{\alpha-1}
    E_{\alpha,\alpha}(-\lambda_k^\beta(t-\tau)^\alpha) \,\bra
    f(\tau),\varphi_k\ket_\LM\dtau
    \qquad(t\ge 0).
  \end{equation}
\end{proposition}

The proposition is proved by an eigenfunction expansion analogously
to~\cite{sakamoto2011initial}. As we use spectral theoretical approach
to consider direct and inverse problem for fractional power operator
$(-\Dg)^\beta$, instead of theory of integral operators, we provide
the detailed proof for the convenience of the reader.

The proof is split in several steps, starting with uniqueness of the
solution, which holds without any assumptions on the source function
$f$.
\begin{proposition}
  \label{prop:FDE-uniqueness}
  There exists at most one strong solution of the fractional diffusion
  equation~\eqref{eq:FDE}.
\end{proposition}
\begin{proof}
  Suppose ${u,\widetilde{u}}:[0,\infty)\to\LM$ are two strong
  solutions of the fractional diffusion equation~\eqref{eq:FDE}, and
  define $v:=u-\widetilde{u}$. Then $v$ is a strong solution
  of~\eqref{eq:FDE} with the zero source term.

  Fix $k\in\Zp:=\{1,2,3,\ldots\}$ and define a complex-valued function
  $v_k:[0,\infty)\to\C$ by $v_k(t):=\bra v(t),\varphi_k\ket_\LM$. It
  is evident that $v_k\in C^1([0,\infty))$ and
  $v_k' = \bra v',\varphi_k\ket_\LM$, because
  $v\in C^1([0,\infty);\LM)$ by the definition of a strong
  solution. If $0<\alpha<1$, combining the previous result with
  Proposition~\ref{prop:integral} shows that
  \begin{equation*}
    \bra \D v(t),\varphi_k\ket_\LM
    = \frac{1}{\Gamma(1-\alpha)}\int_0^t(t-\tau)^{-\alpha}v_k'(\tau)\dtau
    = \D v_k(t)
    \qquad(t\ge 0).
  \end{equation*}
  
  Note that $v(t)\in\dom((-\Dg)^\beta)$. Then the definition of
  $(-\Dg)^\beta$ implies that
  \begin{equation*}
    \bra(-\Dg)^\beta v(t),\varphi_k\ket_\LM = \lambda_k^\beta v_k(t)
    \qquad(t\ge 0).
  \end{equation*}
  Above considerations show that
  \begin{equation}
    \label{eq:v_k-scalar-FDE}
    \begin{aligned}
      \D v_k(t) + \lambda_k^\beta v_k(t) &= 0, \qquad (t\ge 0)\\
      v_k(0) &= 0.
    \end{aligned}
  \end{equation}

  By Proposition~\ref{prop:scalar-FDE} the unique continuous solution
  of~\eqref{eq:v_k-scalar-FDE} is the zero function. Because $k\in\Zp$
  is arbitrary, $v=0$, and therefore $u=\widetilde{u}$.
\end{proof}

Following lemma provides useful estimates for the component functions
$u_k$ of $u$.

\begin{lemma}
  \label{lemma:u-k-estimate}
  Let $T>0$ and $f\in C_c^2((0,T);\LM)$, and let $u_k$ be defined
  by~\eqref{eq:u-k}. Then the following hold:
  \begin{enumerate}[(i)]
  \item The functions $u_k$ satisfy $u_k\in C^1([0,\infty))$, and
    there exists a constant $\epsilon>0$ such that
    $\supp u_k\subset(\epsilon,\infty)$, for every $k\in\Zp$.
  \item For every $k\in\Zp$, it holds that 
    \begin{equation}
      \
      \label{eq:u-k-equality}
      \lambda_k^\beta u_k(t) = \int_0^t
      \big(1-E_{\alpha,1}(-\lambda_k^\beta(t-\tau)^\alpha)\big)\bra
      f'(\tau),\varphi_k\ket_\LM\dtau
      \qquad (t\ge 0),
    \end{equation}
    and
    \begin{equation}
      \label{eq:u-k-inequality}
      \lambda_k^{2\beta}|u_k(t)|^2
      \le \min\{t,T\}\int_0^t |\bra f'(\tau),\varphi_k\ket_\LM|^2\dtau
      \qquad (t\ge 0).
    \end{equation}
  \item If in \eqref{eq:u-k-equality} and \eqref{eq:u-k-inequality}
    the functions $u_k$ and $f'$ are replaced by $u_k'$ and $f''$,
    respectively, \eqref{eq:u-k-equality} and
    \eqref{eq:u-k-inequality} remain valid.
  \end{enumerate}
\end{lemma}
\begin{proof}
  \begin{enumerate}[(i)]
  \item If $\epsilon>0$ is small enough so that $f(t)=0$ for
    $0\le t\le2\epsilon$, then $u_k(t)=0$ for $0\le
    t\le2\epsilon$. Therefore $\supp u_k\subset(\epsilon ,\infty)$.
    
    Define
    \begin{equation}
      \label{eq:F_k}
      F_k(t):=
      \begin{cases}
        t^{\alpha-1}E_{\alpha,\alpha}(-\lambda_k^\beta t^\alpha), & t>0,\\
        0, & t\le 0.
      \end{cases}
    \end{equation}
    Then $F_k$ is locally integrable, and $u_k$ is the convolution
    \begin{equation}
      \label{eq:u-k-convolution}
      u_k(t) = \big(F_k*\bra f,\varphi_k\ket_\LM\big)(t)
      \qquad(t\ge 0).
    \end{equation}
    It follows that $u_k\in C^1([0,\infty))$.
  \item If $\lambda_k=0$, then both sides of~\eqref{eq:u-k-equality}
    vanish, and~\eqref{eq:u-k-inequality} holds trivially. Therefore
    we may assume $\lambda_k>0$.

    Note that as $f(0)=0$, \eqref{eq:F-derivative} and integration by
    parts show that
    \begin{equation*}
      \begin{split}
        \lambda_k^\beta u_k(t)%
        &= \lim_{\epsilon\to 0^+} \int_0^{t-\epsilon}
        \lambda_k^\beta(t-\tau)^{\alpha-1}
        E_{\alpha,\alpha}(-\lambda_k^\beta(t-\tau)^\alpha) \bra
        f(\tau),\varphi_k\ket_\LM\dtau\\
        &= \lim_{\epsilon\to 0^+} E_{\alpha,1}(-\lambda_k^\beta
        \epsilon^\alpha)\bra f(t-\epsilon),\varphi_k\ket_\LM\\
        &\qquad - \int_0^{t-\epsilon}
        E_{\alpha,1}(-\lambda_k^\beta(t-\tau)^\alpha) \bra
        f'(\tau),\varphi_k\ket_\LM\dtau\\
        &= \bra f(t),\varphi_k\ket_\LM - \int_0^t
        E_{\alpha,1}(-\lambda_k^\beta(t-\tau)^\alpha) \bra
        f'(\tau),\varphi_k\ket_\LM\dtau.
      \end{split}
    \end{equation*}
    By combining above with the fact that
    \begin{equation}
      \label{eq:f-ODE}
      \bra f(t),\varphi_k\ket_\LM
      = \int_0^t \bra f'(\tau),\varphi_k\ket_\LM\dtau,
    \end{equation}
    we obtain~\eqref{eq:u-k-equality}.

    Set $t^*:=\min\{{t,T}\}$. The Cauchy--Schwarz inequality applied
    to~\eqref{eq:u-k-equality} shows that
    \begin{equation}
      \label{eq:C-S-to-u-k-equality}
      \lambda_k^{2\beta}|u_k(t)|^2
      \le \int_0^{t^*}|1-E_{\alpha,1}(-\lambda_k^\beta(t-\tau)^\alpha)|^2\dtau
      \int_0^t |\bra f'(\tau),\varphi_k\ket_\LM|^2\dtau.
    \end{equation}

    The function
    \begin{equation*}
      (0,\infty)\ni\tau \mapsto M(\tau) 
      := E_{\alpha,1}(-\lambda_k^\beta\tau^\alpha)
    \end{equation*}
    is completely monotonic, meaning that
    $(-1)^k (\frac{d}{d\tau})^k M(\tau)\ge 0$ for $k=0,1,2,\ldots$ and
    $\tau>0$ (for $\alpha=1$ this is immediate from differentiating
    the exponential function, for $0<\alpha<1$ see Theorem~7.3 in~\cite{MR2680847}). In particular
    \begin{equation*}
      0\le M(\tau)\le M(0)=1
      \qquad(\tau>0).
    \end{equation*}
    It follows that the first integrand
    in~\eqref{eq:C-S-to-u-k-equality} has values in $[0,1]$, and
    {(ii)} is proved.

  \item From~\eqref{eq:u-k-convolution} and properties of convolution
    it follows that~\eqref{eq:u-k} holds if $u_k$ and $f$ are
    substituted by $u_k'$ and $f'$, respectively. This implies that
    \eqref{eq:u-k-equality} and \eqref{eq:u-k-inequality} also hold
    under the same substitution.
  \end{enumerate}
\end{proof}

Next two lemmas prepare for the proof of
Proposition~\ref{prop:existence-and-uniqueness}.
\begin{lemma}
  \label{lemma:u-is-well-defined}
  The series~\eqref{eq:u} converges in $\LM$ for every $t\ge 0$. The
  limit function $u$ is in $C^1([0,\infty);\LM)$, and
  \begin{equation}
    \label{eq:u-derivative}
    u'(t) = \sum_{k=1}^\infty u'_k(t)\varphi_k
    \qquad(t\ge 0),
  \end{equation}
  where the convergence is pointwise in $\LM$. Moreover,
  $u(t)\in\dom((-\Dg)^\beta)$ for every $t\ge 0$, and
  $\supp u\subset(0,\infty)$.
\end{lemma}
\begin{proof}
  Fix $t\ge 0$. Inequality~\eqref{eq:u-k-inequality} implies that
  \begin{equation}
    \label{eq:u-k-ineq-2}
    \sum_{k=1}^\infty \lambda_k^{2\beta}|u_k(t)|^2
    \le t \int_0^t \|f'(\tau)\|_\LM^2\dtau<\infty.
  \end{equation}
  Because $\lambda_k\to\infty$ as $k\to\infty$,
  inequality~\eqref{eq:u-k-ineq-2} implies that
  $\sum_{k=1}^\infty|u_k(t)|^2<\infty$, and therefore the
  series~\eqref{eq:u} converges in $\LM$. It also follows that
  $u(t)\in\dom((-\Dg)^\beta)$.

  The first part of Lemma~\ref{lemma:u-k-estimate} implies that
  $\supp u\subset(0,\infty)$. Therefore on a neighborhood of the
  origin $u$ is smooth and~\eqref{eq:u-derivative} holds.
  
  Because $u_k\in C^1([0,\infty))$, the partial sums of~\eqref{eq:u}
  satisfy $(\sum_{k=1}^N u_k\varphi_k)' = \sum_{k=1}^N
  u_k'\varphi_k$. Hence to prove $u\in C^1([0,\infty))$
  and~\eqref{eq:u-derivative}, it is enough to prove that the series
  on the right-hand side of~\eqref{eq:u-derivative} converges
  uniformly on every subinterval $(0,T)\subset(0,\infty)$.

  For $N$ large enough so that $\lambda_N\ge 1$, {(iii)} of
  Lemma~\ref{lemma:u-k-estimate} yields
  \begin{equation*}
    \sum_{k=N}^\infty |u'_k(t)|^2 \le T\int_0^T \sum_{k=N}^\infty |\bra f''(\tau),\varphi_k\ket_\LM|^2\dtau
    \qquad(0\le t<T).
  \end{equation*}
  The integrand converges to zero pointwise as $N\to\infty$, and it is
  dominated by the integrable function $\|f''\|_\LM^2$. By the
  Lebesgue's dominated convergence theorem, the integral tends to zero
  as $N\to\infty$. This implies uniform convergence of
  $\sum_{k=1}^\infty u_k'\varphi_k$ on $(0,T)$.
\end{proof}

\begin{lemma}
  \label{lemma:u-fractional-derivative}
  The Caputo derivative of order $\alpha\in(0,1]$ of the $\LM$-valued
  function $u$ defined by~\eqref{eq:u} and~\eqref{eq:u-k} exists on
  $[0,\infty)$, and
  \begin{equation}
    \label{eq:u-fractional-derivative}
    \D u(t)
    = \sum_{k=1}^\infty [\D u_k(t)]\varphi_k
    = \sum_{k=1}^\infty \big(-\lambda_k^\beta u_k(t) + \bra f(t),\varphi_k\ket_\LM\big) \varphi_k
    \qquad (t\ge 0).
  \end{equation}
\end{lemma}
\begin{proof}
  By Lemma~\ref{lemma:u-is-well-defined} we have
  $u\in C^1([0,\infty);L^2(M))$, hence the Caputo derivative of order
  $\alpha$ of $u$ exists at every point $t\ge 0$.

  If $\alpha=1$, the first equality
  of~\eqref{eq:u-fractional-derivative} is true by
  Lemma~\ref{lemma:u-is-well-defined}. If $0<\alpha<1$, the first
  equality follows from an application of {(ii)} of
  Proposition~\ref{prop:integral} and~\eqref{eq:u-derivative} in the
  definition of $\D$.

  The second equality follows from Proposition~\ref{prop:scalar-FDE}.
\end{proof}

Proving the existence and uniqueness of a strong solution
of~\eqref{eq:FDE} is now straightforward:
\begin{proof}[Proof of
  Proposition~\ref{prop:existence-and-uniqueness}]
  Proposition~\ref{prop:FDE-uniqueness} implies that a strong
  solution, should it exist, is
  unique. Lemma~\ref{lemma:u-is-well-defined} proves that the function
  $u$ specified by~\eqref{eq:u} and~\eqref{eq:u-k} is a well-defined
  function with range in $\dom((-\Dg)^\beta)$ and $u(0)=0$. By
  Lemma~\ref{lemma:u-fractional-derivative} the Caputo derivative of
  order $\alpha$ of $u$ exists on $[0,\infty)$, and
  \begin{equation*}
    \D u(t) = -(-\Dg)^\beta u(t) + f(t)
    \qquad (t\ge 0).
  \end{equation*}
  Therefore a strong solution exists, and the solution is given
  by~\eqref{eq:u}.
\end{proof}

\subsection{The local source-to-solution operator
  \texorpdfstring{$\LSS$}{L\_V}}
\label{sec:LSS}

Given $f\in C^2_c((0,\infty);\LM)$, let $u^f\in C^1([0,\infty);\LM)$
denote the strong solution of the fractional diffusion
equation~\eqref{eq:FDE}.

\begin{proposition}
  \label{prop:uf-bounded}
  Let $T>0$. There exists a constant $C_{T,M,\alpha}>0$ (that depends on $T$, $\alpha$, and the manifold $(M,g)$) such that
  \begin{equation}
    \label{eq:u-f-inequality}
    \sup_{t\ge0}\|u^f(t)\|_\LM \le C_{T,M,\alpha} \left(\int_0^T \|f'(\tau)\|_\LM^2\dtau\right)^{1/2},
  \end{equation}
  for every $f\in C^2_c((0,T);\LM)$.
\end{proposition}
\begin{proof}
  We have the representation of $u^f$ given by~\eqref{eq:u}
  and~\eqref{eq:u-k}. Let us first estimate the first term $u_1$ of
  the representation.

  Since $\lambda_1=0$, for every $t\ge 0$ it holds that
  \begin{equation}
    \label{eq:u-1-estimate}
    |u_1(t)|
    \le\frac{1}{\Gamma(\alpha)}\sup_{\tau\ge0}|\bra f(\tau),\varphi_1\ket_\LM|
    \,\int_0^{\min\{t,T\}} (t-\tau)^{\alpha-1}\dtau.
  \end{equation}
  The inner product in~\eqref{eq:u-1-estimate} can be estimated
  with~\eqref{eq:f-ODE}. Applying the Cauchy--Schwarz inequality
  to~\eqref{eq:f-ODE} and noticing that the integral
  in~\eqref{eq:u-1-estimate} as a function of $t$ obtains its maximum
  at $t=T$ show that
  \begin{equation}
    \label{eq:u-1}
    |u_1(t)|^2
    \le\frac{1}{\Gamma(\alpha)^2}\int_0^T |\bra f'(\tau),\varphi_1\ket_\LM|^2\dtau\,
    \,\frac{T^{2\alpha+1}}{\alpha^2}.
  \end{equation}

  If $k>1$, then $\lambda_k\ge\lambda_2>0$. Therefore {(ii)}
  of Lemma~\ref{lemma:u-k-estimate} can be used to estimate
  $|u_k(t)|^2$. These estimates together with~\eqref{eq:u-1} readily
  yield~\eqref{eq:u-f-inequality}.
\end{proof}

Recall that in the inverse problem we consider sources supported on an
open set $V\subset M$, and observe the evolution of the corresponding
solutions of the fractional diffusion equation~\eqref{eq:FDE} on the
set $V$. From this information we want to recover the manifold
$(M,g)$.

In what follows, we identify $\LV$ as a subset of $\LM$ by identifying
functions with their zero extensions. Also, by an abuse of notation,
$u^f|_V$ denotes the function $[0,\infty) \ni t\mapsto
u^f(t)|_V$. Then Proposition~\ref{prop:existence-and-uniqueness} and
Proposition~\ref{prop:uf-bounded} imply that if
$f\in C^2_c((0,\infty);\LV$), then
$u^f|_V \in C^1([0,\infty);\LV)\cap L^\infty([0,\infty);\LV)$.

\begin{definition}
  Let $V\subset M$ be a nonempty open set with smooth boundary. The
  \emph{local source-to-solution operator} on $V$, denoted by $\LSS$,
  is the operator
  \begin{equation*}
    \LSS : C_c^2((0,\infty);\LV) \to C^1([0,\infty);\LV)\cap L^\infty([0,\infty);\LV)
  \end{equation*}
  defined by
  \begin{equation*}
    \LSS f := u^f|_V.
  \end{equation*}
  For $T>0$, the \emph{truncated local source-to-solution operator} on
  $V$, denoted by $\LSST$, is the operator
  \begin{equation*}
    \LSST : C_c^2((0,\infty);\LV) \to C^1([0,T);\LV)\cap L^\infty([0,T);\LV)
  \end{equation*}
  defined by $\LSST f := (\LSS f)|_{[0,T)}$.
\end{definition}

Note that as a topological vector space, the space $\LV$ is
independent of the Riemannian metric $g|_V$. Therefore the domain and
codomain of $\LSS$ and $\LSST$ do not depend on $g$.

If $(\varphi_k)_{k=1}^\infty$ and $(\lambda_k)_{k=1}^\infty$ are as in
Section~\ref{sec:FDE-solution}, the local source-to-solution operator
can be represented as
\begin{equation}
  \label{eq:LSS}
  \LSS f(t)
  = \sum_{k=1}^\infty \left[ \int_0^t (t-\tau)^{\alpha-1}
    E_{\alpha,\alpha}(-\lambda_k^\beta(t-\tau)^\alpha)
    \bra
    f(\tau),\varphi_k\ket_\LV\dtau\right]\varphi_k|_V,
\end{equation}
where the sum converges in $\LV$, for every $t\ge 0$.

As a consequence of Proposition~\ref{prop:uf-bounded}, we obtain the
following continuity result for the local source-to-solution operator:
\begin{proposition}
  \label{prop:L-continuity}
  Let $T>0$. Suppose that
  \begin{equation*}
    f\in C_c^2((0,T);\LV)
    \text{ and }
    (f_k)_{k=1}^\infty\subset C_c^2((0,T);\LV)
  \end{equation*}
  are such that $f_k'(t)\to f'(t)$ in $\LV$ as $k\to\infty$, for every
  $t\in(0,T)$, and
  \begin{equation*}
    \sup_{\substack{k\in\Zp,\\t\in(0,T)}} \|f_k'(t)\|_\LV<\infty.
  \end{equation*}
  Then
  \begin{equation*}
    \LSS f_k(t)\to\LSS f(t)\text{ in }\LV\text{ as } k\to\infty,
  \end{equation*}
  uniformly in $t\ge 0$.
\end{proposition}
\begin{proof}
  From the definition of the local source-to-solution operator and
  inequality~\eqref{eq:u-f-inequality} of
  Proposition~\ref{prop:uf-bounded}, it follows that for every
  $t\ge 0$ 
  \begin{equation}
    \label{eq:LSS-inequality}
    \begin{split}
      \|\LSS f(t)-\LSS f_k(t)\|_\LV
      &\le\|u^{f-f_k}(t)\|_\LM\\
      &\le
      C_{T,M,\alpha}\left(\int_0^T\|(f-f_k)'(\tau)\|_\LM^2\dtau\right)^{1/2}.
    \end{split}
  \end{equation}
  By assumption $f_k'(\tau)\to f'(\tau)$ in $\LV$ as $k\to\infty$, for
  every $\tau\in(0,T)$, and the same holds for their zero extensions
  in $\LM$. 
  
  Because the integrand is uniformly bounded with respect to
  $\tau\in(0,T)$ and $k\in\Zp$, the Lebesgue's dominated convergence
  theorem can be applied. This concludes the proof.
\end{proof}

\section{Analysis of the inverse problem}
\label{sec:inverse-problem}

We begin by showing that the local source-to-solution operator can be
determined with a single measurement, provided the source is chosen
appropriately. After that we show that the manifold is determined up
to a Riemannian isometry by this operator.

\subsection{The local source-to-solution operator
  \texorpdfstring{$\LSS$}{L\_V} can be determined with one
  measurement}
\label{sec:LSS-uniqueness}

We construct a source $h\in C_c^\infty((0,T);\LV)$ such that the local
source-to-solution operator $\LSS$ is completely determined by the
single function $\LSST h$.

\begin{definition}
  \label{def:h}
  Fix a constant $T>0$ and let $V\subset M$ be a nonempty open set
  with smooth boundary. Choose a number $0<S<T$, a nonzero
  non-negative function $n\in C_c^\infty(-1,1)$, a bounded sequence
  $(\psi_k)_{k=1}^\infty\subset\LV$ of functions that spans a dense
  subspace of $\LV$, and a sequence $(r(k))_{k=1}^\infty \subset\Zp$
  which contains every positive integer infinitely many times. Then
  define the source $h$ by 
  \begin{equation}
    \label{eq:h}
    h(t) := \sum_{k=1}^\infty%
    \underbrace{ \frac{S^k }%
      {2^{k(k+2)}m_k}\cdot n\left(2^{k+1}\left(\frac{t}{S}-1\right)+3\right)%
    }_{\phantom{MMM}=:h_k(t)}%
    \psi_{r(k)}
    \qquad(t\in\R),
  \end{equation}
  where $m_k:=\max_{0\le l\le k} \|n^{(l)}\|_\infty$.
\end{definition}
We provide illustration of the decay of functions $h_k$ as $k$ increases in Figure~\ref{fig:hk}. 
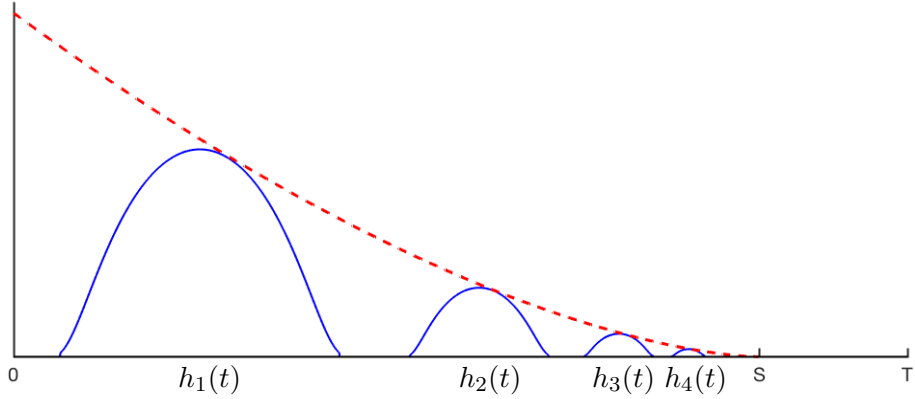
\begin{figure}\label{fig:hk}
\begin{picture}(250,130)
\put(10,0){\includegraphics[width=0.95\textwidth]{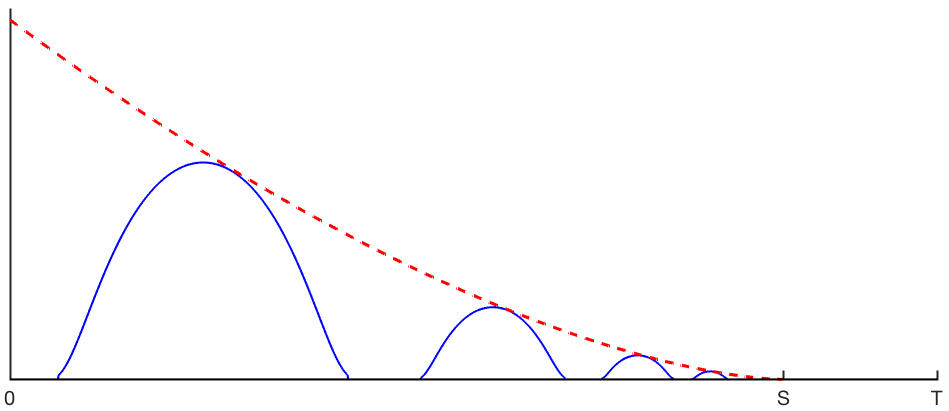}}
\put(75,0){$h_1(t)$}
\put(180,0){$h_2(t)$}
\put(230,0){$h_3(t)$}
\put(257,0){$h_4(t)$}
\end{picture}
\caption{Examples of functions $h_k$, $k\geq 1$, are plotted in blue with the choice of $n$ being a Gaussian bump. The exponential decay of the amplitude of $h_k$ as $k$ increases is illustrated by the red dashed line. The values of $h_k$ are not presented in scale.}
\end{figure}

\begin{remark}
  Boundedness and denseness in $\LV$ are qualities that are
  independent of the metric $g|_V$, because all Riemannian metrics on
  $M$ induce the same topology on $\LV$.
\end{remark}
\begin{remark}
  An example of a sequence that contains every positive integer
  infinitely many times is the sequence that begins
  \begin{equation*}
    1, 1, 2, 1, 2, 3, 1, 2, 3, 4, \ldots
  \end{equation*}
\end{remark}

In order to prove Proposition~\ref{prop:LSS-uniqueness}, consider a
manifold $(\widetilde{M},\widetilde{g})$ and an open set
$\widetilde{V}\subset\widetilde{M}$ with smooth boundary, and suppose they
satisfy the conditions of Theorem~\ref{thm:main-theorem}. Let
$\theta:\cl(\widetilde{V})\to\cl(V)$ be a diffeomorphism. Because of
the compactness of $\cl(\widetilde{V})$ and $\cl(V)$, the pullback
$\theta^*$ is a continuous operator from $L^2(\cl(V))$ onto
$L^2(\cl(\widetilde{V}))$. By the diffeomorphism invariance of the
boundary, $\theta^*$ is also a continuous operator from $\LV$ onto
$L^2(\widetilde{V})$. Therefore, if $h\in C^\infty_c((0,T);\LV)$, then
$\theta^*h\in C^\infty_c((0,\infty);L^2(\widetilde{V}))$.

It is convenient to introduce the conjugated operators $\LSSTilde$ and
$\LSSTTilde$ defined for $f\in C^2_c((0,\infty);\LV)$ by
\begin{equation*}
  \LSSTilde f := (\theta^*)\inv\LSSop_{\widetilde{V}}(\theta^* f)
  \text{ and }
  \LSSTTilde f := (\theta^*)\inv\LSSop_{\widetilde{V},T}(\theta^* f).
\end{equation*}
Then
\begin{equation*}
  \LSSTilde :  C_c^2((0,\infty);\LV) \to C^1([0,\infty);\LV)\cap L^\infty([0,\infty);\LV)
\end{equation*}
and
\begin{equation*}
  \LSSTTilde : C_c^2((0,\infty);\LV) \to C^1([0,T);\LV)\cap L^\infty([0,T);\LV),
\end{equation*}
and~\eqref{eq:LSSh-equality} and~\eqref{eq:LSS-equality} are
equivalent to $\LSST h = \LSSTTilde h$ and $\LSS = \LSSTilde$,
respectively.

Proposition~\ref{prop:LSS-uniqueness} will be proved in several
steps. Let us first prove the compactness and smoothness of $h$.

\begin{proposition}
  The terms $h_k\in C_c^\infty((0,\infty))$ in the series~\eqref{eq:h}
  satisfy
  \begin{equation}
    \label{eq:supp-of-h_k}
    \supp h_k \subset \big( (1-2^{1-k})S, (1-2^{-k})S \big)
    \qquad(k\in\Zp).
  \end{equation}
  In particular, their supports are pairwise disjoint. Furthermore,
  the series converges uniformly in $t\in\R$, and defines a function
  $h\in C_c^\infty((0,T);\LV)$.
\end{proposition}

\begin{proof}
  Inclusion~\eqref{eq:supp-of-h_k}, which is seen to hold by a
  straightforward calculation, implies that $h$ is defined pointwise
  and supported in $(0,S]\subset(0,T)$. For $l=0,1,2,\ldots$, the
  $l$-th derivative $h_k^{(l)}$ with $k$ large enough so that $k\ge l$
  and $2^{k+1}\ge S$ can be estimated as
  \begin{equation}
    \label{eq:h-derivative-estimate}
    |h_k^{(l)}(t)| =
    \frac{2^{(k+1)l}}{S^l}\frac{S^k|n^{(l)}(2^{k+1} t/S-2^{k+1}+3)|}%
    {2^{k(k+2)}\max_{0\le l\le k} \|n^{(l)}\|_\infty}%
    \le \frac{1}{2^k} \qquad(t\in\R).
  \end{equation}
  Therefore $\sum_{k=1}^\infty \|h_k^{(l)}\|_\infty < \infty$, for
  every $l=0,1,2,\ldots$ Estimate~\eqref{eq:h-derivative-estimate}
  together with Proposition~\ref{prop:general-prop} stated in the
  Appendix implies the remaining claims.
\end{proof}

\begin{lemma}
  \label{lemma:L-time-invariance}
  If $f\in C_c^2((0,\infty);\LV)$ and $t_0\in\R$ is a such that
  $f_{t_0} := f(\cdot-t_0)\in C_c^2((0,\infty);\LV)$, then
  \begin{equation*}
    \LSS f_{t_0}(t) =
    \begin{cases}
      0, & 0\le t<t_0,\\
      \LSS f(t-t_0), & t\ge\max\{0,t_0\}.
    \end{cases}
  \end{equation*}
\end{lemma}
\begin{proof}
  This follows from a straightforward change of variables
  in~\eqref{eq:LSS}.
\end{proof}
  
Following proposition states the essential fact that if the support of
the source $f$ is included in the time interval $(0,T')$, then the
future evolution of $u^f|_V$ is determined completely by its evolution
up to time $T'$.
\begin{proposition}
  \label{prop:L-holomorphic}
  Let $T'>0$ and consider $f\in C^2_c((0,T'); \LV)$. If
  $\LSSop_{V,T'} f = \widetilde{\LSSop}_{\widetilde{V},T'}f$, then
  $\LSS f = \LSSTilde f$.
\end{proposition}
\begin{proof}
  The proposition will be proved by extending $\LSS f$
  holomorphically onto a region of the complex plane and applying the
  uniqueness of holomorphic continuation. For properties of
  vector-valued holomorphic functions, we refer to~\cite{MR1157815}.

  Fix $\epsilon>0$ small enough so that
  $\supp f\subset(0,T'-2\epsilon)$. We show that the $\LV$-valued
  mapping
  \begin{equation*}
    \R\supset (T'-\epsilon,\infty)\ni t\mapsto \LSS f(t) \in\LV
  \end{equation*}
  extends holomorphically onto the complex region
  $\{z\in\C:\re(z)>T'-\epsilon\}$. This is enough to prove the
  claim. Namely, in this case also $\LSSTilde f$ extends
  holomorphically onto the region, and by assumption the extensions
  agree on $(T'-\epsilon,T')$. By uniqueness of holomorphic extension,
  they agree everywhere on the region, so in particular also on
  $(T'-\epsilon,\infty)$.

  Consider complex-valued functions $g_k$ on $\C_+\times[0,\infty)$,
  where $k\in\Zp$, defined by
  \begin{equation*}
    g_k(z,\tau):= \bra f(\tau),\varphi_k\ket_\LM
    (z+T'-\epsilon-\tau)^{\alpha-1} E_{\alpha,\alpha}(-\lambda_k^\beta(z+T'-\epsilon-\tau)^\alpha).
  \end{equation*}
  By assumption $f(\tau)=0$ if $\tau\ge T'-2\epsilon$, therefore the
  functions are well-defined. In addition, an inspection shows that
  \begin{equation}
    \label{eq:L-extension}
    \LSS f(z+T'-\epsilon)
    = \sum_{k=1}^\infty \int_0^{T'} g_k(z,\tau)\dtau\,\varphi_k|_V
    \qquad (z\in(0,\infty)).
  \end{equation}
  We show that the right-hand side of~\eqref{eq:L-extension} is an
  $\LV$-valued holomorphic function on $\C_+$.

  By {(i)} of Proposition~\ref{prop:ML-properties},
  $E_{\alpha,\alpha}$ is bounded on $\C\setminus\C_+$. It follows that
  with a constant $C = C(\alpha,\epsilon)>0$ we have
  \begin{equation}
    \label{eq:gk-estimate}
    |g_k(z,\tau)|\le
    C |\bra f(\tau),\varphi_k\ket_\LM|\qquad(z\in\C_+,\,\tau\ge 0,\,k\ge 1).
  \end{equation}
  Consequently
  \begin{equation}
    \label{eq:D-k}
    \left|\int_0^{T'} g_k(z,\tau)\dtau\right|^2
    \le C^2T'\int_0^{T'} |\bra f(\tau),\varphi_k\ket_\LM|^2\dtau
    \qquad(z\in\C_+,\, k\ge 1).
  \end{equation}

  Let $D_k\ge 0$ denote the right-hand side of~\eqref{eq:D-k}. Then
  for all $N>M>0$ and $z\in\C_+$ it holds that
  \begin{equation*}
    \Bigg\|\sum_{k=1}^N \int_0^{T'} g_k(z,\tau)\dtau\,\varphi_k -
    \sum_{k=1}^M \int_0^{T'} g_k(z,\tau)\dtau\,\varphi_k\Bigg\|^2_\LM
    \le\sum_{k=M+1}^N D_k,
  \end{equation*}
  and
  \begin{equation*}
    \sum_{k=1}^\infty D_k
    = C^2T'\int_0^{T'} \|f(\tau)\|^2_\LM\dtau
    <\infty.
  \end{equation*}
  It follows from the Cauchy criterion for uniform convergence that
  for every $z\in\C_+$ the series
  \begin{equation}
    \label{eq:int-gk}
    \sum_{k=1}^\infty \int_0^{T'} g_k(z,\tau)\dtau\,\varphi_k 
  \end{equation}
  converges in the topology of $\LM$, and that the convergence is
  uniform in $z\in\C_+$.

  Using the fact that for every $\tau\ge 0$ the function
  $g_k(\cdot,\tau)$ is holomorphic on $\C_+$ , it is straightforward
  to verify with Morera's theorem and Fubini's theorem that the
  function
  \begin{equation*}
    \C_+\ni z\mapsto \int_0^{T'} g_k(z,\tau)\dtau\in \C
  \end{equation*}
  is also holomorphic. As a uniform limit of holomorphic functions,
  the $\LM$-valued function defined on $\C_+$ by the
  series~\eqref{eq:int-gk} is holomorphic. Consequently also the
  function
  \begin{equation}
    \label{eq:int-G-restricted}
    \C_+\ni z\mapsto \sum_{k=1}^\infty \int_0^{T'} g_k(z,\tau)\dtau\,\varphi_k|_V \in\LV
  \end{equation}
  is holomorphic. Because~\eqref{eq:int-G-restricted}
  extends~\eqref{eq:L-extension} from $(0,\infty)$ onto $\C_+$, the
  proof is finished.
\end{proof}

The following two results prepare for the proof of
Proposition~\ref{prop:LSS-uniqueness}.

\begin{lemma}
  \label{lemma:L-on-h-psi}
  If
  \begin{equation}
    \label{eq:Lh-assumption}
    \LSST h = \LSSTTilde h,
  \end{equation}
  then $\LSS(h_k\psi_{r(k)}) = \LSSTilde(h_k\psi_{r(k)})$ for every
  $k\in\Zp$.
\end{lemma}
\begin{proof}
  Assume that~\eqref{eq:Lh-assumption} holds. Consider an integer
  $j\ge 0$ and for induction purposes assume that
  $\LSS(h_k\psi_{r(k)})=\LSSTilde(h_k\psi_{r(k)})$ for $1\le k\le
  j$. For $j=0$ this is vacuously true.

  Suppose $j>0$ and let
  $T':=(1-2^{-(j+1)})S<T$. Inclusion~\eqref{eq:supp-of-h_k} implies
  that for every $l>j+1$ the function $h_l\psi_{r(l)}$ vanishes on
  $(0,T')$. This implies that
  \begin{equation}
    \label{eq:Lh-j+1-equality}
    \sum_{k=1}^{j+1}\LSSop_{V,T'}(h_k\psi_{r(k)})
    = \LSSop_{V,T'} h,
  \end{equation}
  and an analogous equality holds for
  $\widetilde{\LSSop}_{\widetilde{V},T'}$. Now
  equalities~\eqref{eq:Lh-assumption} and~\eqref{eq:Lh-j+1-equality}
  and the fact that $T'<T$ imply
  \begin{equation}
    \label{eq:j+1-equality}
    \sum_{k=1}^{j+1}\LSSop_{V,T'}(h_k\psi_{r(k)})
    = \sum_{k=1}^{j+1}\widetilde{\LSSop}_{\widetilde{V},T'}(h_k\psi_{r(k)}).
  \end{equation}
  
  By induction hypothesis the first $j$ terms in the sums
  of~\eqref{eq:j+1-equality} agree, therefore the last terms have to
  agree, also. This, and Proposition~\ref{prop:L-holomorphic} imply
  $\LSS(h_{j+1}\psi_{r(j+1)})=\LSSTilde(h_{j+1}\psi_{r(j+1)})$, and
  the induction is finished.
\end{proof}

\begin{lemma}
  \label{lemma:L-a-xi}
  If $\LSST h = \LSSTTilde h$, then $\LSS(a\xi)=\LSSTilde(a\xi)$ for
  every $a\in C^2_c((0,\infty))$ and $\xi\in\LV$.
\end{lemma}
\begin{proof}
  The first step is to show that
  \begin{equation}
    \label{eq:L-convolution-equality}
    \LSS(a\psi_l) = \LSSTilde(a\psi_l)
    \qquad(l\in\Zp,\,a\in C^2_c((0,\infty))).
  \end{equation}
  For that purpose, let us fix an integer $l\in\Zp$ and a function
  $a\in C_c^2((0,\infty))$, and choose a constant $\delta=\delta(a)>0$
  such that $a(t)=0$ for $t\in(-\infty,\delta]$.

  For every $k\in\Zp$, define a scaled translate $d_k$ of $h_k$ by
  setting
  \begin{equation*}
    d_k(t) := \frac{h_k(t + S)}{ \int_\R h_k(\tau)\dtau}
    \qquad(t\in\R).
  \end{equation*}
  Then $d_k$ is a non-negative function and following hold:
  \begin{equation}
    \label{eq:properties-of-dk}
    \int_\R d_k(\tau)\dtau=1
    \text{ and }
    \supp d_k \subset \left(-\frac{S}{2^{k-1}},-\frac{S}{2^k}\right)
    \qquad(k\in\Zp).
  \end{equation}
  It follows that the sequence $(d_k*a)_{k=1}^\infty$ of convolutions
  and the sequence $((d_k*a)')_{k=1}^\infty$ of their derivatives
  satisfy
  \begin{equation}
    \label{eq:a-dk-convcergence}
    d_k*a(t) \to a(t)\text{ and }(d_k*a)'(t)\to a'(t)\text{ as } k\to\infty
    \qquad(t\in\R),
  \end{equation}
  where both convergences are uniform in $t\in\R$.

  From the inclusion of the support of $d_k$
  in~\eqref{eq:properties-of-dk}, it follows that
  \begin{equation}
    \label{eq:supp-of-ad-k}
    a(t_0)d_k(\cdot - t_0) \in C_c^\infty((0,\infty))
    \qquad(2^{k-1}\ge S/\delta,\,t_0\in\R).
  \end{equation}
  From~\eqref{eq:supp-of-ad-k}, Lemma~\ref{lemma:L-time-invariance},
  and Lemma~\ref{lemma:L-on-h-psi}, it follows that
  \begin{equation}
    \label{eq:LSS-a-dk-psi}
    \LSS(a(t_0)d_k(\cdot-t_0)\psi_{r(k)})
    =\LSSTilde(a(t_0)d_k(\cdot-t_0)\psi_{r(k)})
    \quad(2^{k-1}\ge S/\delta,\,t_0\in\R).
  \end{equation}

  By the assumed property of the sequence $r$, there exists some
  $k\in\Zp$ such that $2^{k-1}\ge S/\delta$ and $r(k)=l$. Choose any
  such $k$, and define a sequence
  $(b_m)_{m=1}^\infty\in C^\infty_c((0,\infty))$ of functions by
  setting
  \begin{equation}
    \label{eq:Riemann-sum}
    b_m(t) := \frac{1}{m}\sum_{j\in\Z} a(j/m)d_k(t - j/m)
    \qquad(t\in\R,\,m\in\Zp). 
  \end{equation}
  Note that the sum is always in fact finite, and
  by~\eqref{eq:LSS-a-dk-psi} it holds that
  \begin{equation}
    \label{eq:bm-equality}
    \LSS(b_m\psi_l) = \LSSTilde(b_m\psi_l)
    \qquad(m\in\Zp).
  \end{equation}

  The Riemann sums~\eqref{eq:Riemann-sum} satisfy
  \begin{equation}
    \label{eq:uniform-convergence}
    b_m(t) \to d_k * a(t)
    \text{ and }
    b'_m(t) \to (d_k * a)'(t)
    \text{ as } m\to\infty
    \qquad(t\in\R),
  \end{equation}
  uniformly in $t\in\R$ (see, e.g.,\ Lemma~4.1.3 in
  \cite{MR1065993}). Because all supports of the functions $b_m$ as
  well as the support of $d_k*a$ are included in some bounded
  interval, the uniform convergence~\eqref{eq:uniform-convergence}
  together with Proposition~\ref{prop:L-continuity} and
  equality~\eqref{eq:bm-equality} imply that in $\LV$ it holds that
  \begin{equation}
    \label{eq:L-on-convolution}
    \LSS (d_k * a\,\psi_l)
    = \lim_{m\to\infty} \LSS (b_m \psi_l)
    = \lim_{m\to\infty} \LSSTilde (b_m\psi_l)
    = \LSSTilde (d_k * a\,\psi_l).
  \end{equation}
  Once more we use the assumed property of the sequence $r$ to choose
  an increasing sequence $(k_j)_{j=1}^\infty$ of indices such that
  $2^{k_1-1}\ge S/\delta$ and $r(k_j)=l$ for every
  $j\in\Zp$. Then~\eqref{eq:L-on-convolution} holds for every index
  $k_j$. If we let $j\to\infty$, using~\eqref{eq:a-dk-convcergence}
  and the same reasoning as above, we obtain
  \begin{equation*}
    \LSS (a\,\psi_l)
    = \lim_{j\to\infty} \LSS (d_{k_j}*a \psi_l)
    = \lim_{j\to\infty} \LSSTilde (d_{k_j}*a\psi_l)
    = \LSSTilde (a\,\psi_l).
  \end{equation*} 

  Finally we can use the denseness of
  $\linspan\{\psi_l:l\in\Zp\}\subset\LV$
  and~\eqref{eq:L-convolution-equality} to conclude by
  Proposition~\ref{prop:L-continuity}
  that~\eqref{eq:L-convolution-equality} holds also if $\psi_l$ is
  replaced by an arbitrary function $\xi\in\LV$.
\end{proof}

We are now ready to prove Proposition~\ref{prop:LSS-uniqueness}:
\begin{proof}[Proof of Proposition~\ref{prop:LSS-uniqueness}]
  We need to prove that $\LSS = \LSSTilde$.

  Pick $f\in C^2_c((0,\infty);\LV)$ and let
  $(\xi_k)_{k=1}^\infty\subset\LV$ be an orthonormal basis. Define
  \begin{equation*}
    f_N(t):=\sum_{k=1}^N\bra f(t),\xi_k\ket_\LV\xi_k
    \qquad(t\in\R,\,N\in\Zp).
  \end{equation*}
  Then for $T'>0$ large enough it holds that
  \begin{align*}
    (f_N)_{N=1}^\infty&\subset C_c^2((0,T');\LV),\\
    \sup_{\substack{N\in\Zp,\\t\in(0,T')}} \|f_N'(t)\|_\LV &\le \sup_{t\in\R}\|f'(t)\|_\LV<\infty,\text{ and}\\
    f_N'(t)&\to f'(t)\text{ as }N\to\infty,\text{ for every } t\in\R.
  \end{align*}
  Consequently the conditions of Proposition~\ref{prop:L-continuity}
  hold, and therefore for every $t\ge 0$ we have
  \begin{align*}
    \LSS f(t)
    = \lim_{N\to\infty} \LSS f_N(t)
    = \lim_{N\to\infty} \LSSTilde f_N(t)
    =\LSSTilde f(t).
  \end{align*}
  Here the convergence is in the topology of $\LV$, the middle
  equality is due to Lemma~\ref{lemma:L-a-xi} and the linearity of
  $\LSS$ and $\LSSTilde$, and other equalities are due to
  Proposition~\ref{prop:L-continuity}. Because
  $f\in C^2_c((0,\infty);\LV)$ is arbitrary, the proof is finished.
\end{proof}

\subsection{The local source-to-solution operator
  \texorpdfstring{$\LSS$}{L\_V} determines the manifold}
\label{sec:M-uniqueness}

Let $(\lambda_k)_{k=1}^\infty$ and $(\varphi_k)_{k=1}^\infty$ be as in
Section~\ref{sec:FDE-solution}. Let $(q_k)_{k=1}^\infty\subset\Zp$ be
a sequence such that $(\lambda_{q_k})_{k=1}^\infty$ is the strictly
increasing sequence that contains all distinct eigenvalues of $-\Dg$,
and let $E_k\subset\LM$ be the eigenspace corresponding to the
eigenvalue $\lambda_{q_k}$. Furthermore, let $P_k:\LM\to\LM$ be the
orthogonal projection onto $E_k$, and define $P_{V,k}:\LV\to\LV$ by
\begin{equation*}
  P_{V,k}u := (P_ku)|_V
  \qquad(u\in\LV\subset\LM).
\end{equation*}
Then $P_k\in B(\LM)$ and $P_{V,k}\in B(\LV)$, where the sets are the
spaces of bounded linear operators on $\LM$ and $\LV$,
respectively. We consider these spaces as normed spaces with the
operator norm.

Suppose that $\varphi_{K+1},\varphi_{K+2},\ldots,\varphi_{K+\dim E_k}$
is the subsequence of the orthonormal basis $(\varphi_k)_{k=1}^\infty$
that spans the eigenspace $E_k$. Then
\begin{equation*}
  P_k u = \sum_{k=K+1}^{K+\dim E_k} \bra u,\varphi_k\ket_\LM\varphi_k
  \qquad(u\in\LM).
\end{equation*}
It follows from~\eqref{eq:LSS} that the local source-to-solution
operator can be written as
\begin{equation*}
  \label{eq:L}
  \LSS f(t) = \sum_{k=1}^\infty \int_0^t (t-\tau)^{\alpha-1}
  E_{\alpha,\alpha}(-\lambda_{q_k}^\beta(t-\tau)^\alpha) P_{V,k}
  f(\tau)\dtau
  \qquad(t\ge 0),
\end{equation*}
where the sum converges in $\LV$, for every $t\ge 0$.

\begin{proposition}
  \label{prop:H}
  Consider the region $\Omega:=\C\setminus(-\infty,0]$ and the
  $B(\LV)$-valued mapping $H_V$ on $\Omega$ defined by
  \begin{equation}
    \label{eq:H_V}
    H_V(z) := \sum_{k=1}^\infty\frac{1}{z + \lambda_{q_k}^\beta}P_{V,k}
    \qquad(z\in\Omega).
  \end{equation}
  Then the following hold:
  \begin{enumerate}
  \item For every $z\in\Omega$ the series~\eqref{eq:H_V} converges in
    $B(\LV)$ in the operator norm topology, and the $B(\LV)$-valued
    function $H_V$ is holomorphic on $\Omega$.
  \item For every $z_0\in\C$ the following limit holds:
    \begin{equation}
      \label{eq:limit-H}
      \lim_{\substack{z\to -z_0,\\z\in\Omega}}(z+z_0)H_V(z) =
      \begin{cases}
        0, & z_0\notin\{\lambda_{q_k}^\beta:k\in\Zp\},\\
        P_{V,k}, & z_0 = \lambda_{q_k}^\beta.
      \end{cases}
    \end{equation}
  \end{enumerate}
\end{proposition}
Note that as every complex number is a limit point of $\Omega$, the
limit~\eqref{eq:limit-H} can be considered also for
$z_0\in[0,\infty)$.
\begin{proof}
  In what follows, it is convenient to explicitly write out the zero
  extension and restriction operators. Thus, let $Z:\LV\to\LM$ and
  $Z^* :\LM\to\LV$ be the operators that extend a function with zero
  from $V$ to $M$, and restrict a function on $M$ to $V$,
  respectively.

  \begin{enumerate}
  \item It follows from $\lim_{k\to\infty}\lambda_{q_k}^\beta=\infty$
    and the fact that the operators $P_k$ project onto mutually
    orthogonal subspaces, that the function
    \begin{equation}
      \label{eq:H}
      \Omega
      \ni z \mapsto
      H(z):=\sum_{k=1}^\infty\frac{1}{z + \lambda_{q_k}^\beta}P_k
      \in B(\LM)
    \end{equation}
    is holomorphic on $\Omega$. As
    \begin{equation}
      \label{eq:H-V}
      H_V(z) = Z^*H(z)Z,
    \end{equation}
    the function $H_V$ is holomorphic on $\Omega$, as well.
  \item Suppose that $z_0\notin\{\lambda_{q_k}^\beta:k\in\Zp\}$. Then
    there exists a constant $\delta>0$ such that if $|z+z_0|<\delta$,
    then
    \begin{equation}
      \label{eq:lower-bound}
      |z+\lambda_{q_k}^\beta|>\delta
      \qquad(k\in\Zp).
    \end{equation}
    It is easy to see that for every $z\in\Omega$ for which
    \eqref{eq:lower-bound} holds, also
    \begin{equation*}
      \|(z+z_0)H(z)\|_{B(\LM)} < \frac{|z+z_0|}{\delta}
    \end{equation*}
    holds. Letting $z\to -z_0$ in $\Omega$ and using the boundedness
    of $Z$ and $Z^*$ in~\eqref{eq:H-V} prove the first case
    of~\eqref{eq:limit-H}.
    
    If $z_0=\lambda_{q_k}^\beta$, we can write
    \begin{equation}
      \label{eq:H_V-2}
      (z+\lambda_{q_k}^\beta)H_V(z)
      =  P_{V,k} + (z+\lambda_{q_k}^\beta)\sum_{\substack{l=1,\\l\neq k}}^\infty\frac{1}{z + \lambda_{q_l}^\beta}P_{V,l}.
    \end{equation}
    As $z\to-\lambda_{q_k}^\beta$ in $\Omega$, the right-hand side
    of~\eqref{eq:H_V-2} tends to $P_{V,k}$ by the same reasoning as
    above.\qedhere
  \end{enumerate}
\end{proof}

The following proposition relates the function $H_V$ to the Laplace
transforms of $\LSS f$ and $f$. It plays an essential part in the
proof of Theorem~\ref{thm:M-equality}:
\begin{proposition}
  \label{prop:H-Laplace}
  Let $H_V$ be as in Proposition~\ref{prop:H}. For every source
  $f\in C_c^2((0,\infty);\LV)$ and complex number $s\in\C_+$, the
  Laplace transform of $f$ and $\LSS f$ exist at the point $s$, and
  they are related by the equality
  \begin{equation}
    \label{eq:LSS-Laplace}
    \mathcal{L}\LSS f(s) = H_V(s^\alpha)\mathcal{L}f(s)
    \qquad(s\in\C_+).
  \end{equation}
\end{proposition}
\begin{proof}
  In this proof it is also convenient to use the extension and
  restriction operators $Z$ and $Z^*$ from the proof of
  Proposition~\ref{prop:H}.

  Let $f\in C^2_c((0,\infty);\LV)$ and let
  \begin{equation*}
    u^{Zf}\in C^1([0,\infty);\LM)\cap L^\infty([0,\infty);\LM)
  \end{equation*}
  be the strong solution of~\eqref{eq:FDE} with the source
  $Zf\in C^2_c((0,\infty);\LM)$. The boundedness of the function
  $u^{Zf}$ implies that for every $s\in\C_+$, the $\LM$-valued
  function $[0,\infty)\ni\tau\mapsto e^{-s\tau}u^{Zf}(\tau)$ is
  integrable, so that the Laplace transform
  \begin{equation*}
    \mathcal{L}u^{Zf}(s) := \int_0^\infty e^{-s\tau}u^{Zf}(\tau)\dtau
  \end{equation*}
  is defined for all $s\in\C_+$. Due to the compact support of $f$, an
  analogous reasoning shows that the Laplace transform
  $\mathcal{L}Zf(s)$ is defined for all $s\in\C$.

  We can write
  \begin{equation*}
    u^{Zf}(t) = \sum_{k=1}^\infty \big(F_k*\bra Zf,\varphi_k\ket_\LM\big)(t)\varphi_k
    \qquad(t\ge 0),
  \end{equation*}
  where the functions $F_k$ are defined in~\eqref{eq:F_k}. By
  Proposition~\ref{prop:integral}, we can take the Laplace transform
  of $u^{Zf}$ componentwise. With {(ii)} of
  Proposition~\ref{prop:ML-properties} and $H$ as defined
  in~\eqref{eq:H}, this results in
  \begin{equation*}
    \begin{split}
      \mathcal{L} u^{Zf}(s)
      &= \sum_{k=1}^\infty \mathcal{L}\big(F_k*\bra Zf,\varphi_k\ket_\LM\big)(s)\varphi_k\\
      &= \sum_{k=1}^\infty\frac{1}{s^\alpha+\lambda_k^\beta}
      \bra\mathcal{L}Zf(s),\varphi_k\ket_\LM\varphi_k\\
      &= \sum_{k=1}^\infty\frac{1}{s^\alpha+\lambda_{q_k}^\beta}P_k\mathcal{L}Zf(s)\\
      &= H(s^\alpha)\mathcal{L}Zf(s) \qquad(s\in\C_+).
    \end{split}
  \end{equation*}
  Using the fact that $\mathcal{L}$ commutes with $Z$ and $Z^*$, we
  obtain
  \begin{equation*}
    \begin{split}
      \mathcal{L}\LSS f(s)
      &= \mathcal{L}Z^*u^{Zf}(s)\\
      &= Z^* H(s^\alpha)\mathcal{L}Zf(s)\\
      &= H_V(s^\alpha)\mathcal{L}f(s) \qquad(s\in\C_+).\qedhere
    \end{split}
  \end{equation*}
\end{proof}

To consider Theorem~\ref{thm:M-equality}, let
$(\widetilde{M},\widetilde{g})$, $\widetilde{V}\subset\widetilde{M}$,
$\theta:\cl(\widetilde{V})\to\cl(V)$, $\LSSop_{\widetilde{V}}$, and
$\LSSTilde$ be as in Section~\ref{sec:LSS-uniqueness}. Let
$(\widetilde{\lambda}_k)_{k=1}^\infty\subset[0,\infty)$ be the
sequence of eigenvalues of $-\Delta_{\widetilde{g}}$ (counted with
multiplicities), and let $(\widetilde{q}_k)_{k=1}^\infty\subset\Zp$ be
a sequence such that
$(\widetilde{\lambda}_{\widetilde{q}_k})_{k=1}^\infty$ is the strictly
increasing sequence of all distinct eigenvalues of
$-\Delta_{\widetilde{g}}$. Define operators
$(P_{\widetilde{V},k})_{k=1}^\infty\subset B(L^2(\widetilde{V}))$ and
function $H_{\widetilde{V}}:\Omega\to B(L^2(\widetilde{V}))$
analogously to $(P_{V,k})_{k=1}^\infty\subset B(\LV)$ and
$H_V:\Omega\to B(\LV)$, respectively.

Let
\begin{equation*}
  \widetilde{P}_{\widetilde{V},k}
  :=(\theta^*)^{-1}P_{\widetilde{V},k}\theta^*\in B(\LV)
  \qquad(k\in\Zp)
\end{equation*}
be the conjugated operators, and let $\widetilde{H}_{\widetilde{V}}$
be the pointwise conjugated $B(\LV)$-valued function defined by
\begin{equation*}
  \widetilde{H}_{\widetilde{V}}(z)
  := (\theta^*)^{-1} H_{\widetilde{V}}(z)\theta^*
  =  \sum_{k=1}^\infty\frac{1}{z+\widetilde{\lambda}_{\widetilde{q}_k}^\beta} \widetilde{P}_{\widetilde{V},k}
  \qquad(z\in\Omega).
\end{equation*}

\begin{proposition}
  \label{prop:set-equality}
  The local source-to-solution operator $\LSS$ uniquely determines the
  pairs $(\lambda_{q_k},P_{V,k})$, i.e., if $\LSS=\LSSTilde$, then
  \begin{equation}
    \label{eq:equality}
    \{(\lambda_{q_k},P_{V,k}):k\in\Zp\}
    = \{(\widetilde{\lambda}_{\widetilde{q}_k},\widetilde{P}_{\widetilde{V},k}):k\in\Zp\}.
  \end{equation}
\end{proposition}
\begin{proof}
  The pointwise conjugated function $\widetilde{H}_{\widetilde{V}}$ is
  holomorphic on $\Omega$, and the limit~\eqref{eq:limit-H} holds if
  $H_V$, $P_{V,k}$ and $\lambda_{q_k}$ are replaced by
  $\widetilde{H}_{\widetilde{V}}$, $\widetilde{P}_{\widetilde{V},k}$
  and $\widetilde{\lambda}_{\widetilde{q}_k}$, respectively. Also,
  equality~\eqref{eq:LSS-Laplace} holds if $\LSS$ and $H_V$ are
  replaced by $\LSSTilde$ and $\widetilde{H}_{\widetilde{V}}$,
  respectively.

  Fix a nonzero non-negative function $a(t)\in C^2_c((0,\infty))$ and
  pick an arbitrary function $\xi\in\LV$. If
  $\LSS(a(t)\xi)=\LSSTilde(a(t)\xi)$, then
  \begin{equation}
    \label{eq:H_V-equality}
    \begin{split}
      \big(\mathcal{L}a(s)\big)H_V(s^\alpha)\xi
      &= \mathcal{L}\big(\LSS(a(t)\xi)\big)(s)\\
      &= \mathcal{L}\big(\LSSTilde(a(t)\xi)\big)(s)\\
      &=
      \big(\mathcal{L}a(s)\big)\widetilde{H}_{\widetilde{V}}(s^\alpha)\xi
      \qquad(s\in\C_+),
    \end{split}
  \end{equation}
  where the first and last equality are due to~\eqref{eq:LSS-Laplace}.
  Because $\xi\in\LV$ is arbitrary and $\mathcal{L}a(s)>0$ for every
  $s\in\R_+$, equality~\eqref{eq:H_V-equality} implies
  \begin{equation}
    \label{eq:H_V-equality-2}
    H_V(s) = \widetilde{H}_{\widetilde{V}}(s)
    \qquad(s\in\R_+).
  \end{equation}
  
  Proposition~\ref{prop:H} states that both sides
  of~\eqref{eq:H_V-equality-2} are holomorphic functions on
  $\Omega$. Because $\Omega$ is a region and the functions agree on
  $(0,\infty)$, they must agree everywhere on $\Omega$.
  
  Due to the unique continuation principle, note that the functions
  $(\varphi_k|_V)_{k=1}^\infty$ are linearly independent, and
  therefore $P_{V,k}\neq 0$. From~\eqref{eq:limit-H} and the fact that
  $H_V=\widetilde{H}_{\widetilde{V}}$ on $\Omega$ it follows that
  \begin{equation*}
    \begin{split}
      P_{V,k}
      &=\lim_{\substack{s\to-\lambda_{q_k}^\beta,\\s\in\Omega}} (s+\lambda_{q_k}^\beta)H_V(s)\\
      &=\lim_{\substack{s\to-\lambda_{q_k}^\beta,\\s\in\Omega}} (s+\lambda_{q_k}^\beta)\widetilde{H}_{\widetilde{V}}(s)\\
      &=
      \begin{cases}
        0, &  \lambda_{q_k}^\beta\notin\{\widetilde{\lambda}^\beta_{\widetilde{q}_l}:l\in\Zp\},\\
        \widetilde{P}_{\widetilde{V},l}, & \lambda_{q_k}^\beta =
        \widetilde{\lambda}_{\widetilde{q}_l}^\beta.
      \end{cases}
    \end{split}
  \end{equation*}
  Thus there must exist an index $l\in\Zp$ such that
  $\lambda_{q_k}=\widetilde{\lambda}_{\widetilde{q}_l}$ and
  $P_{V,k}=\widetilde{P}_{\widetilde{V},l}$.

  We have shown that the left-hand side of~\eqref{eq:equality} is a
  subset of the right-hand side. The other direction follows from
  symmetry.
\end{proof}

In order to prove Theorem~\ref{thm:M-equality}, we reduce the
situation from the fractional diffusion equation~\eqref{eq:FDE} to
that of the wave equation on the same manifold $(M,g)$:
\begin{equation}
  \label{eq:wave}
  \begin{aligned}
    (\partial_t^2-\Dg) w(x,t) &= p(x,t),  && (x,t)\in M\times (0,\infty),\\
    w(x,0) &= 0, && x\in M,\\
    \partial_t w(x,0) &= 0, && x\in M.
  \end{aligned}
\end{equation}

For a source $p\in C^\infty_c(V\times(0,\infty))$, let
$w^p\in C^\infty(M\times(0,\infty))$ denote the unique solution
of~\eqref{eq:wave}, and define the hyperbolic local source-to-solution
operator
$\LSS^{\mathrm{hyp}}:C^\infty_c(V\times(0,\infty))\to
C^\infty(V\times(0,\infty))$ by
\begin{equation*}
  \LSS^{\mathrm{hyp}}p := w^p|_{V\times(0,\infty)}.
\end{equation*}

\begin{proof}[Proof of Theorem~\ref{thm:M-equality}]
  Consider $p\in C_c^\infty(V\times(0,\infty))$. The solution $w^p$ of
  the wave equation~\eqref{eq:wave} can be written as
  \begin{equation}
    \label{eq:wave-solution}
    w^p(x,t) = \sum_{k=1}^\infty\int_0^t s_k(t-\tau)\bra p(\cdot,\tau),\varphi_k\ket_\LM\varphi_k(x)\dtau
    \qquad(t\ge 0),
  \end{equation}
  where
  \begin{equation*}
    s_1(t):=t\text{ and } s_k(t):=\frac{\sin(\sqrt{\lambda_k}\,t)}{\sqrt{\lambda_k}},\text{ for }k \ge 2,
  \end{equation*}
  and the series~\eqref{eq:wave-solution} converges in
  $L^2(M\times[0,T])$, for every $T>0$ (see Corollary~2
  of~\cite{MR3826551}). Thus
  \begin{equation}
    \label{eq:LSS-hyp}
    \LSS^{\mathrm{hyp}}p(x,t) = \sum_{k=1}^\infty \int_0^ts_{q_k}(t-\tau) P_{V,k}\,p\,(x,\tau)\dtau
    \qquad(t\ge 0),
  \end{equation}
  where the series converges in $L^2(V\times[0,T])$, for every $T>0$.

  By Proposition~\ref{prop:set-equality}, the local source-to-solution
  operator $\LSS$ for the fractional diffusion equation~\eqref{eq:FDE}
  uniquely determines the eigenvalues
  $(\lambda_{q_k})_{k=1}^\infty\subset[0,\infty)$ of $-\Dg$ and the
  corresponding operators $(P_{V,k})_{k=1}^\infty\subset
  B(\LV)$. From~\eqref{eq:LSS-hyp} it follows that these in turn
  uniquely determine the hyperbolic local source-to-solution operator
  $\LSS^{\mathrm{hyp}}$ for the wave equation~\eqref{eq:wave}. By
  Theorem~2 of~\cite{MR3826551}, $\LSS^{\mathrm{hyp}}$ determines the
  manifold $(M,g)$ up to a Riemannian isometry. The proof is finished.
\end{proof}

\begin{proof}[Proof of Theorem~\ref{thm:main-theorem}]
  Let $h\in C_c^\infty((0,T);\LV)$ be a source as defined in
  Definition~\ref{def:h}, and suppose~\eqref{eq:main-theorem} of
  Theorem~\ref{thm:main-theorem} holds with this
  source. Then~\eqref{eq:LSSh-equality} of
  Proposition~\ref{prop:LSS-uniqueness} holds, and
  consequently~\eqref{eq:LSS-equality} holds for every source
  $f\in C^2_c((0,\infty);\LV)$.  Therefore, by
  Theorem~\ref{thm:M-equality}, the manifolds $(M,g)$ and
  $(\widetilde{M},\widetilde{g})$ are Riemannian isometric.
\end{proof}

\appendix

\section{Calculus of Hilbert space valued functions of a real
  variable}
\label{sec:calculus}

Consider a measurable subset $I\subset\R$ and a function
$y:I\to\LM$. We say that $y$ is \emph{integrable} if it is strongly
measurable and $\int_I\|y(\tau)\|_\LM\dtau<\infty$. The measure on $I$
is the Lebesgue measure, and measurability of the norm is a
consequence of the strong measurability of $y$. For $1\le p\le\infty$,
the space $L^p(I;\LM)$ consists of those strongly measurable functions
$y:I\to\LM$ for which $\|y\|_\LM\in L^p(I)$. We recall that for
$y\in L^1(I;\LM)$ the (Bochner) integral $\int_I y(\tau)\dtau\in\LM$
is defined. For general theory of integration of functions with values
in a Banach space, we refer the reader  to~\cite{MR1216137}.

Suppose then that $I\subset\R$ is an interval with at least two points
and fix a point $t\in I$. We recall that $y$ is said to be
\emph{differentiable at the point $t$}, if there exists a function
$\xi\in\LM$ such that
\begin{equation*}
  \lim_{h\to 0} \Big\|\frac{y(t+h)-y(t)}{h} - \xi\Big\|_\LM = 0.
\end{equation*}
If $t$ is an endpoint of $I$, the limit is the appropriate one-sided
limit.

The \emph{derivative of $y$ at $t$}, denoted by $y'(t)$, is defined to
be the function $\xi\in\LM$. If $y$ is differentiable at every point
of $I$ and the so obtained function $y':I\to\LM$ is continuous, $y$ is
\emph{continuously differentiable}. The space of all continuously
differentiable functions is denoted by $C^1(I;\LM)$. Higher order
derivatives and spaces $C^k(I;\LM)$ are defined recursively exactly as
in the case of scalar functions.

\begin{proposition}
  \label{prop:integral}
  Consider a function $y:I\to\LM$,
  $y(t)=\sum_{k=1}^\infty y_k(t)\psi_k$, where $I\subset\R$ is measurable, $y_k:I\to\C$ are
  complex-valued functions, $(\psi_k)_{k=1}^\infty\subset\LM$ is an
  orthonormal basis, and the series converges in $\LM$ for every
  $t\in I$. Then the following hold:
  \begin{enumerate}[(i)]
  \item The $\LM$-valued function $y$ is strongly measurable, if and
    only if all the complex-valued functions $y_k:I\to\C$ are measurable.
  \item If $y\in L^1(I;\LM)$, then
    \begin{equation}
      \label{eq:componentwise-integral}
      \int_I y(\tau)\dtau = \sum_{k=1}^\infty\left[\int_I y_k(\tau)\dtau\right] \psi_k.
    \end{equation}
  \end{enumerate}
\end{proposition}

Note that the integrals on the right-hand side
of~\eqref{eq:componentwise-integral} are ordinary Lebesgue integrals
of complex-valued functions.
\begin{proof}
  If $y$ is strongly measurable and $k\in\Zp$, the component function
  $y_k=\bra y,\psi_k\ket_\LM$ is strongly measurable as the
  composition of a continuous function with a strongly measurable
  function. A scalar strongly measurable function on $I$ is
  measurable.

  On the other hand, if $y_k:I\to\C$ is measurable, it is an almost
  everywhere limit of complex-valued step functions, and therefore the
  $\LM$-valued map $y_k\psi_k:I\ni t\mapsto y_k(t)\psi_k\in\LM$ is an
  almost everywhere limit of $\LM$-valued step maps. In other words,
  $y_k\psi_k$ is strongly measurable. It follows that the partial sums
  $\sum_{k=1}^N y_k\psi_k$ are strongly measurable, and therefore $y$
  as their pointwise limit in $\LM$ is strongly
  measurable.

  If $y\in L^1(I;\LM)$, it is a property of the integral that
  \begin{equation*}
    \left\bra\int_I y(\tau)\dtau,\xi\right\ket_\LM
    = \int_I \bra y(\tau),\xi\ket_\LM\dtau
    \qquad(\xi\in\LM).
  \end{equation*}
  Applying this with $\xi=\psi_k$
  proves~\eqref{eq:componentwise-integral}.
\end{proof}

\begin{proposition}
  \label{prop:general-prop}
  Let $X$ be a Banach space, $(\xi_k)_{k=1}^\infty\subset X$ be a bounded
  sequence, and $(h_k)_{k=1}^\infty\subset C^\infty(I)$ be a sequence of
  complex-valued functions, where $I\subset\R$ is an open set.  Suppose that
  the derivatives $h_k^{(l)}$ satisfy
  \begin{equation}
    \label{eq:M-condition}
    \sum_{k=1}^\infty \|h_k^{(l)}\|_\infty < \infty
    \qquad(l=0,1,2,\ldots).
  \end{equation}
  Then the series $f:=\sum_{k=1}^\infty h_k\xi_k$ converges uniformly
  on $I$, and $f\in C^\infty(I; X)$.
  Furthermore, the derivatives of
  $f$ are obtained by term-wise differentiation, and also those series
  converge uniformly on $I$.
\end{proposition}

\begin{proof}
  The terms $ h_k\xi_k:I\to X$ are continuous, and by the Weierstrass
  M-test (using the boundedness of $\xi_k$
  and~\eqref{eq:M-condition}), the series $\sum_{k=1}^\infty h_k\xi_k$
  converges uniformly to $f$. It follows that $f$ is continuous.

  For $N\in\Zp$ we have
  \begin{equation*}
    \left(\sum_{k=1}^N  h_k\xi_k\right)'=\sum_{k=1}^N h_k'\xi_k,
  \end{equation*}
  and the same reasoning as above implies that
  \begin{equation*}
    \left(\sum_{k=1}^N h_k\xi_k\right)' \to \sum_{k=1}^\infty h_k'\xi_k
  \end{equation*}
  uniformly as $N\to\infty$. It follows from standard results of
  differentiation (see, e.g.,~\cite{MR1216137}, Theorem~9.1) that $f$ is
  differentiable and $f'=\sum_{k=1}^\infty h_k'\xi_k$.

  An easy induction finishes the proof.
\end{proof}

\section*{Acknowledgement}
The research was been partially supported by Academy of Finland, grants 273979, 284715, 312110, 314879 and the Atmospheric mathematics project of University of Helsinki.

\bibliographystyle{abbrv} \bibliography{interior-source}

\end{document}